\documentclass[11pt,twoside,leqno,a4paper,final]{amsart}
\usepackage[a4paper=true,pdfpagelabels,colorlinks]{hyperref} 
\usepackage{amsmath}
\usepackage{mathtools,nccmath}

\usepackage{amssymb}
\usepackage{enumitem}

\usepackage[dvipsnames]{xcolor}

\usepackage{mathrsfs} 

\theoremstyle{plain}
\newtheorem{proposition}{Proposition}[section]
\newtheorem{corollary}[proposition]{Corollary}
\newtheorem{lemma}[proposition]{Lemma}
\newtheorem{theorem}[proposition]{Theorem}

\theoremstyle{definition}
\newtheorem{definition}[proposition]{Definition}

\newtheorem{remark}[proposition]{Remark}
\newtheorem{remarks}[proposition]{Remarks}

\newcommand{\B}{\mathscr{B}}
\newcommand{\BB}{\mathcal{B}}

\newcommand{\C}{{\mathbb C}}

\newcommand{\D}{{\mathbb D}}

\renewcommand{\H}{\mathcal{H}}

\newcommand{\LL}{\mathcal{L}}

\newcommand{\N}{{\mathbb N}}

\newcommand{\R}{{\mathbb R}}
\newcommand{\Q}{{\mathbb Q}}
\newcommand{\T}{\mathbb{T}}
\newcommand{\Z}{{\mathbb Z}}

                  \def\z{\zeta}

\usepackage{bm}

\newcommand{\boldnorm}[1]{\pmb{[}#1\pmb{]}}

\makeatletter
\newcommand{\opnorm}{\@ifstar\@opnorms\@opnorm}
\newcommand{\@opnorms}[1]{%
  \left|\mkern-1.5mu\left|\mkern-1.5mu\left|
   #1
  \right|\mkern-1.5mu\right|\mkern-1.5mu\right|
}
\newcommand{\@opnorm}[2][]{%
  \mathopen{#1|\mkern-1.5mu#1|\mkern-1.5mu#1|}
  #2
  \mathclose{#1|\mkern-1.5mu#1|\mkern-1.5mu#1|}
}
\makeatother

\makeatletter
\newcommand{\boldopnorm}{\@ifstar\@boldopnorms\@boldopnorm}
\newcommand{\@boldopnorms}[1]{%
  \pmb{\left|}\mkern-1.5mu\pmb{\left|}\mkern-1.5mu\pmb{\left|}
   #1
  \pmb{\right|}\mkern-1.5mu\pmb{\right|}\mkern-1.5mu\pmb{\right|}
}
\newcommand{\@boldopnorm}[2][]{%
  \mathopen{#1\pmb{|}\mkern-1.5mu#1\pmb{|}\mkern-1.5mu#1\pmb{|}}
  #2
  \mathclose{#1\pmb{|}\mkern-1.5mu#1\pmb{|}\mkern-1.5mu#1\pmb{|}}
}
\makeatother

\numberwithin{equation}{section}

\theoremstyle{plain} 
\newcommand{\thistheoremname}{}
\newtheorem*{genericthm*}{\thistheoremname}
\newenvironment{namedthm*}[1]
{\renewcommand{\thistheoremname}{#1}%
	\begin{genericthm*}}
	{\end{genericthm*}}

\theoremstyle{definition} 
\newcommand{\thisdefinitionname}{}
\newtheorem*{genericdefinition*}{\thisdefinitionname}
\newenvironment{nameddefinition*}[1]
{\renewcommand{\thisdefinitionname}{#1}%
	\begin{genericdefinition*}}
	{\end{genericdefinition*}}

\makeatletter
\@namedef{subjclassname@2020}{\textup{2020} Mathematics Subject Classification}
\makeatother

\begin{document}

\title [Words of analytic paraproducts 
on Hardy and Bergman spaces]
{Words of analytic paraproducts 
on Hardy and weighted Bergman spaces} 

\author[A. Aleman]{Alexandru Aleman}
\address{A. Aleman: Department of Mathematics, University of Lund, P.O.\ Box 118, SE-221 00, Lund, Sweden}
\email{alexandru.aleman@math.lu.se}
\author[C. Cascante]{Carme Cascante}
\address{C. Cascante: Departament de Matem\`atiques i
    Inform\`atica, Universitat  de Barcelona,
     Gran Via 585, 08071 Barcelona, Spain \&
Centre de Recerca Matem\`atica, Edifici C, Campus Bellaterra, 08193 Bellaterra, Spain}
\email{cascante@ub.edu}
\author[J. F\`abrega]{Joan F\`abrega}
\address{J. F\`abrega: Departament de Matem\`atiques i
    Inform\`atica, Universitat  de Barcelona,
     Gran Via 585, 08071 Barcelona, Spain}
\email{joan$_{-}$fabrega@ub.edu}
\author[D. Pascuas]{Daniel Pascuas}
\address{D. Pascuas: Departament de Matem\`atiques i
    Inform\`atica, Universitat  de Barcelona,
     Gran Via 585, 08071 Barcelona, Spain}
\email{daniel$_{-}$pascuas@ub.edu}
\author[J. A. Pel\'aez]{Jos\'e \'Angel Pel\'aez}
\address{J. A. Pel\'aez: Departamento de An\'alisis Matem\'atico, Universidad de M\'alaga, Campus de
Teatinos, 29071 M\'alaga, Spain} \email{japelaez@uma.es}

\thanks{
The research of the second, third and fourth authors
	was supported in part by
        Ministerio de Ciencia e Innovaci\'{o}n, 
        Spain, project  PID2021-123405NB-I00.  The second author is also supported by the Spanish State Research Agency, through the Severo Ochoa and Mar\'{\i}a de Maeztu Program for Centers and Units of Excellence in R\&D 
(CEX2020-001084-M), and Generalitat de Catalunya, 2021SGR00087.
The research of the fifth author was supported in part by La Junta de Andaluc{\'i}a, project FQM210}
\date{\today}



\date{\today}

\begin{abstract}

For a fixed analytic function g on the unit disc, we consider the analytic paraproducts induced by g, which are formally defined by 
$T_gf(z)=\int_0^zf(\zeta)g'(\zeta)d\zeta$,  $S_gf(z)=\int_0^zf'(\zeta)g(\zeta)d\zeta$, and $M_gf(z)=g(z)f(z)$. 
We are concerned with the study of the boundedness of operators in the algebra  $\mathcal{A}_g$ generated by the above operators acting on Hardy, or  standard weighted Bergman spaces on the disc.  The general question is certainly very  challenging, since operators in  $\mathcal{A}_g$ are finite linear combinations of finite products (words) of $T_g,S_g,M_g$  which may involve a large amount of cancellations to be understood.   The results in 
\cite{Aleman:Cascante:Fabrega:Pascuas:Pelaez}  show that  boundedness of operators in a  fairly large subclass of $\mathcal{A}_g$ can be characterized by one of the conditions $g\in H^\infty$, or $g^n$ belongs to $BMOA$ or the Bloch space, for some integer $n>0$. However, it is  also proved that there are many operators, even single words in $\mathcal{A}_g$ whose boundedness  cannot be described in terms of these conditions.
The present paper provides a considerable progress in this direction. Our main result provides a complete quantitative characterization of the boundedness of an arbitrary word in  $\mathcal{A}_g$  in terms of a ``fractional power'' of the symbol $g$, that only depends on the number of appearances of each of the letters  $T_g,S_g,M_g$ in the given word.

\end{abstract}
\maketitle

\noindent
\section{Introduction}
Let $\H(\D)$ be the space of analytic functions on the unit disc $\D$ of the complex plane. For  $\alpha>-1$ and $0<p<\infty$, let $L^p_{\alpha}$ be the space of all complex-valued measurable functions $f$ on $\D$ such that
\begin{equation*}
\|f\|_{\alpha,p}^p:=
(\alpha+1)\int_{\D}|f(z)|^p(1-|z|^2)^{\alpha}\,dA(z)<\infty,
\end{equation*}
where $dA$ is the normalized area measure on $\D$. Then $A^p_{\alpha}:=L^p_{\alpha}\cap\H(\D)$ is the classical weighted Bergman space. As usual, let $H^p$, $0<p\le \infty$, denote the classical Hardy space of analytic functions on $\D$. To simplify the  notations,  we shall write $A^p_{-1}:=H^p$, for $0<p<\infty$. 
We also write $L^p_{-1}=L^p(\T)$ and $\|\cdot\|_{-1,p}=\|\cdot\|_{L^p(\T)}$, for $0<p<\infty$.

Given $g\in \H(\D)$, let us consider the multiplication operator $M_gf= fg$ and the integral operators
\[
\begin{aligned}
	T_gf(z)= \int_0^z f(\z)g'(\z)\,d\z \quad
	S_gf(z)= \int_0^z f'(\z)g(\z)\,d\z \qquad
	(z\in\D).
	\end{aligned}
\]

They satisfy the fundamental identity:
\begin{equation}\label{eqn:formula}
	M_g=T_g+S_g +g(0)\,\delta_0\qquad(g\in\H(\D)),
\end{equation}
where $\delta_0f=f(0)$, for $f\in\H(\D)$.
Following~{\cite{Aleman:Cascante:Fabrega:Pascuas:Pelaez}} we shall call  these 
 integral operators {\em$g$-analytic paraproducts}. It is well-known  
(see again~{\cite{Aleman:Cascante:Fabrega:Pascuas:Pelaez}} and the references therein)   that  $T_g$ is bounded  on $A^p_\alpha$ if and only if $g$ belongs to 
the Bloch space $\B$  when $\alpha>-1$,  and $g\in BMOA$  in the Hardy space case $\alpha=-1$, while the boundedness of  $M_g$ and $S_g$ on any of these spaces is equivalent to the boundedness of $g$ in $\D$, {\em i.e.} $g\in H^\infty$.

 The object of interest is the algebra $\mathcal{A}_g$ generated by these operators. More precisely, we investigate the  boundedness (continuity) on $A^p_{\alpha}$, $p>0$, $\alpha\ge-1$, of the operators formally defined  as finite linear combinations of finite products of $M_g,T_g,S_g$, which we shall also call {\em $g$-operators}.

To continue the discussion we need to recall the so-called $ST$-form of a $g$-operator, that is, 
according to~{\cite[\S 3.2]{Aleman:Cascante:Fabrega:Pascuas:Pelaez}} every $L\in \mathcal{A}_g$ has a representation
\begin{equation*}
L=\sum_{k=0}^nS_g^kT_gP_k(T_g)+S_gP_{n+1}(S_g)+g(0)P_{n+2}(g-g(0))\,\delta_0,
\end{equation*}
where $n\in\N_0:=\N\cup\{0\}$ and all the $P_k$'s are polynomials.  
The approach in~{\cite{Aleman:Cascante:Fabrega:Pascuas:Pelaez}} uses this representation to detect  certain dominant terms  which lead to a characterization of the boundedness of some of such operators.  More precisely,  
{\cite[Theorem 1.2]{Aleman:Cascante:Fabrega:Pascuas:Pelaez}}~shows that: \newline
a) If $P_{n+1}\ne 0$ then $L$ is bounded on $A_\alpha^p$ if and only if $g\in H^\infty$,\newline
b) If $P_{n+1}= 0$, but $P_n(0)\ne 0$ then $L$ is bounded on $A_\alpha^p,~\alpha>-1$, if and only if $g^{n+1}\in \B$,\newline
c) If $\alpha=-1$, $P_{n+1}= 0$ and $P_n$ is a nonzero constant then $L$ is bounded on $A_\alpha^p$ if and only if $g^{n+1}\in BMOA$.

A direct consequence is a  full characterization of the boundedness of all the compositions of two paraproducts 
(see~{\cite[Corollary 1.4]{Aleman:Cascante:Fabrega:Pascuas:Pelaez}}). In this case,
the boundedness involves besides the  conditions $g\in H^\infty$ and $g\in BMOA (\B)$, the third condition $g^2\in BMOA\,\,(\B)$. Nevertheless, the situation changes dramatically for the compositions of three  paraproducts, since the conditions $g^2\in BMOA\,\,(\B)$ and $g\in BMOA\,\,(\B)$  do not necessarily characterize their boundedness. Indeed, if $g(z)=\bigr(\log(\frac{e}{1-z})\bigl)^{7/12}$, then $S_gT_g^2=\frac12T_{g^2} T_g$ is actually compact on $A_\alpha^p$, but clearly $g^2\notin\B$.
Moreover, although $g(z)=\log(\frac{e}{1-z})$ belongs to $BMOA$, $S_gT_g^2$ is not bounded on $A^p_{\alpha}$ 
(see~{\cite[Theorem 1.5]{Aleman:Cascante:Fabrega:Pascuas:Pelaez}}).

The primary aim of this paper is to study the above situation in the general setting. More precisely, we are going to prove a complete quantitative characterization in terms of the symbol $g$ of the boundedness on $A_\alpha^p$ of an arbitrary operator product (composition) with factors $S_g,T_g$, or $M_g$. Such  operators will be called  {\em $g$-words}. Thus an {\em $N$-letter $g$-word} is an operator of the form $L=L_1\cdots L_N$, where each \emph{letter} $L_j$ is either $M_g$, $S_g$ or $T_g$.  For notational purposes we shall fix the number of appearances of each $g$-analytic paraproduct, that is, 
if $\ell,m,n\in\N_0$ satisfy $N=\ell+m+n\ge 1$, we consider the set  $W_g(\ell,m,n)$ consisting of $N$-letter $g$-words $L$ of the form \begin{equation*}
	L=L_1\cdots L_N,
\end{equation*} 
with $\#\{j:L_j=M_g\}=\ell$, $\#\{j:L_j=S_g\}=m$, and 
$\#\{j:L_j=T_g\}=n$. Also let  $W_g(0,0,0)$ consists only of  the identity operator.

In order to state our main result we need to introduce the appropriate classes of symbols. These are natural generalizations of $BMOA$ or the Bloch space $\B$ obtained by replacing in their definitions the modulus of the complex derivative by the Euclidean norm of the gradient of some positive power of the modulus of the function in question. To be more precise, consider the seminorms
\begin{equation*}
	\|g\|^2_{BMOA}:=\sup_{a\in\D}\int_{\D}(1-|\phi_a|^2)|g'|^2\,dA
	\quad\mbox{and}\quad
	\|g\|_{\B}:=\sup_{z\in\D}(1-|z|^2)|g'(z)|.
\end{equation*}
where $\phi_a(z):=\frac{a-z}{1-\bar{a}z},~a,z\in\D$, and recall that 
	\begin{equation*}
	BMOA:=\{g\in\H(\D):\|g\|_{BMOA}<\infty\}\mbox{ and }
	\B:=\{g\in\H(\D):\|g\|_{\B}<\infty\}.
\end{equation*}
For a function $\psi:\D\to\R$ we define 
\begin{equation*}
	|\nabla\psi|(z)
	:=\limsup_{w\to z}\frac{|\psi(w)-\psi(z)|}{|w-z|}\in[0,\infty]
	\qquad(z\in\D).
\end{equation*}
The notation  $|\nabla\psi|(z)$ is justified by the fact that, when $\psi$ is differentiable at $z$, $|\nabla\psi|(z)$ is just the Euclidean norm $|\nabla\psi(z)|$ of the gradient of $\psi$ at $z$.
Note that if $g\in\H(\D)$ then $|\nabla|g||=|g'|$, so that we can write
	\begin{equation*}
	\|g\|^2_{BMOA}:=\sup_{a\in\D}\int_{\D}(1-|\phi_a|^2)|\nabla|g||^2\,dA
	\mbox{ and }
	\|g\|_{\B}:=\sup_{z\in\D}(1-|z|^2)|\nabla|g||(z).
\end{equation*}
The  function classes which are relevant for our purposes are obtained 
by replacing, for any real $q\ge 1$,
$|\nabla|g||$ by $$|\nabla|g|^q|(z)=q|g(z)|^{q-1}|g'(z)|,$$
and the above seminorms by the expressions
\begin{align}
\label{eqn:BMOAq:functional}	
	\|g\|^{2q}_{BMOA^q}&:=\sup_{a\in\D}\int_{\D}(1-|\phi_a|^2)|\nabla|g|^q|^2\,dA,
	\\
\label{eqn:Blochq:functional}	
	\|g\|_{\B^q}^q&:=\sup_{z\in\D}(1-|z|^2)|\nabla|g|^q|(z). 	
\end{align}

Note that while $\|\cdot\|_{BMOA^q}$ always makes sense for $0<q<1$,  $\|\cdot\|_{\B^q}$ does not (see \eqref{eqn:BMOAqnorm:equal:boldnorm:-1} and Remark~{\ref{remarks:symbols}~{\ref{remarks:symbols:1}}}).
Morever, when $q>1$, these quantities do not define seminorms (see Remark~{\ref{remarks:symbols}~{\ref{remarks:symbols:4}}}).
The corresponding classes of analytic functions on the unit disc are defined by
\[
BMOA^q  :=\{g\in\H(\D):\|g\|_{BMOA^q}<\infty\}
\]
 and 
\[
\\\B^q :=\{g\in\H(\D):\|g\|_{\B^q}<\infty\}.
\]

 To simplify notation we write 
 $\B^q_{\alpha}:=\B^q$, for $\alpha>-1$, and 
 $\B^q_{-1}:=BMOA^q$. We are going to thoroughly investigate these classes in Section~{\ref{section:symbols}} below.
  A remarkable property is that these sets of functions are strictly decreasing in  $q$.
  
 The space of bounded linear operators on $A^p_{\alpha}$ is denoted by $\BB(A^p_{\alpha})$ and, for any linear operator $L:\H(\D)\to\H(\D)$, even when $p\in (0,1)$, we write  
 \begin{equation*}
 	\|L\|_{\alpha,p}
 	:=\sup\{\|Lf\|_{\alpha,p}:f\in A^p_{\alpha},\,\|f\|_{\alpha,p}\le1\}.
 \end{equation*}
 Moreover, as usual, $A\lesssim B$ ($B\gtrsim A$) for nonnegative functions $A$, $B$ means that $A\le C\,B$, for some positive constant $C$
 independent of the variables involved. Furthermore, we write $A\simeq B$ when $A\lesssim B$ and $A\gtrsim B$. 

\begin{theorem}
	\label{thm:sharp:estimate:norm:word}
	Let $\alpha\ge-1$, $0<p<\infty$, $g\in\H(\D)$ and 
	$L \in W_g(\ell,m,n)$, where  $\ell,m,n\in\N_0$ and $N=\ell+m+n\ge1$.
	\begin{enumerate}[label={\sf\alph*)},topsep=3pt, 
		leftmargin=0pt, itemsep=4pt, wide, listparindent=0pt, itemindent=6pt] 
		\item If $n=0$, then $L$ is bounded on  $A^p_{\alpha}$ 
		if and only if $g\in H^\infty$ and 
		\begin{equation*}
			\|L\|_{\alpha,p}
			\simeq\|g\|^{N}_{H^\infty}.
		\end{equation*}
		\item If $n\ge1$, then $L$ is bounded on $A^p_{\alpha}$  if and only $g\in \B^s_{\alpha}$, where $s=\frac{\ell+m}n+1$. Moreover,
		\begin{equation*}
			\|L\|_{\alpha,p}
			\simeq\|g\|^{N}_{\B^s_{\alpha}}.
		\end{equation*}
	\end{enumerate}
\end{theorem}

We should note that the boundedness assertion in part a) is covered 
by~{\cite[Theorem 1.2]{Aleman:Cascante:Fabrega:Pascuas:Pelaez}.} The operator-norm estimate requires additional work.
The proof of part b) is long, rather technical and needs a new technique that we briefly describe here. Firstly, we reduce the proof  to the model case $L=S^m_gT^n_g$ by using  algebraic formulas derived from~{\eqref{eqn:formula}}. In order to handle the model case, we relate
 the boundedness of $S^m_gT^n_g$ on $A^p_{\alpha}$ to the boundedness of a fractional  operator defined for fixed parameters $\tau\in\Q$, $\tau>0$ and $\ell\in\N$ by
 $$Q^{\tau,\ell}_gf=|g|^{\tau\ell}T^{\ell}_gf$$ acting from $A_\alpha^p$ into the corresponding $L^p$-space. To be more precise, we prove that 
 \begin{equation*}
 \|S^m_gT^n_g\|_{\alpha,p}\simeq\|Q^{\frac{m}n,n}_g\|_{\alpha,p}
 \qquad(g\in\H(\overline{\D})),
 \end{equation*}
 The proof of the theorem is then completed by showing that 
  \begin{equation*}
 \|Q^{\tau,\ell}_g\|_{\alpha,p}\simeq\|g\|^{(\tau+1)\ell}_{\B^{\tau+1}_{\alpha}}
 \qquad(g\in\H(\overline{\D})).
  \end{equation*} 
There are  several key steps needed to achieve the above goals and their presentation  essentially fills sections \ref{section:easycase}-\ref{section:proofs:of:thm1.3-1.4:proposition1.5} in the paper.
Most of these results are new and of interest in their own right.
The first one is the nested property  $\B^r_{\alpha}\subset \B^q_{\alpha},~1\le q<r$ (see Corollary \ref{cor1:prop:BMOAq:Bq:nested} below) which 
is proved using Garsia-type seminorms for these symbol classes.
Another important tool is the equivalence of the 
 functionals \eqref{eqn:BMOAq:functional} and \eqref{eqn:Blochq:functional} to the norms of  embeddings of $A^p_{\alpha}$ into certain tent spaces (Carleson type-norms), combined with  a classical theorem due to Calder\'{o}n that relates the $A^p_{\alpha}$ norm of a function with the norm of its derivative in a suitable tent space (see \S\ref{subsection:Calderon:thm}). We should point out here that we follow a unified approach which is the same for Bergman and Hardy spaces. However,  for  the Bergman space case there is an alternative (conceptually simpler)
 proof which is essentially based on pointwise estimates and  Littlewood-Paley formula.

  Finally, we  also mention that many lower  estimates of the quantities considered above are based on algebraic properties of iterated commutators of $g$-operators (see \S\ref{subsection:lower:norm:estimate:g-operators}).

An immediate application of the main result reveals an interesting property of  the symbol classes~{$\B^q_{\alpha}$.}
\begin{corollary}
	Let $\alpha\ge-1$ and $m,n\in \N$ such that $m=m_1+\dots+m_k$ and $n=n_1+\dots+n_k$, where  $m_j\in\N_0$  and $n_j\in\N$, for $j=1,\dots,k$.
	Then 
	\begin{equation*}
		\|g\|^{N}_{\B^s_{\alpha}}
		\lesssim
		\|g\|^{N_1}_{\B^{s_1}_{\alpha}}\cdots \|g\|^{N_k}_{\B^{s_k}_{\alpha}}
		\qquad(g\in\H(\D)),  
	\end{equation*}	
	where  $s=\frac{m}n +1$, $N=m+n$, $N_j=m_j+n_j$ and $s_j=\frac{m_j}{n_j}+1$, for $j=1,\dots k$.	
\end{corollary}

Indeed,  Theorem \ref{thm:sharp:estimate:norm:word} gives
	\begin{align*}
		\|g\|^{N}_{\B^s_{\alpha}}
		&\simeq
		\|(S^{m_1}_gT^{n_1}_g)\cdots(S^{m_k}_gT^{n_k}_g)\|_{\alpha,p}
		\\
		&\le
		\|(S^{m_1}_gT^{n_1}_g)\|_{\alpha,p}\cdots\|(S^{m_k}_gT^{n_k}_g)\|_{\alpha,p}\simeq
		\|g\|^{N_1}_{\B^{s_1}_{\alpha}}\cdots \|g\|^{N_k}_{\B^{s_k}_{\alpha}}.\qedhere	
	\end{align*}		

A more involved application of Theorem~{\ref{thm:sharp:estimate:norm:word}} is the following boundedness result of $g$-operators on $A^p_{\alpha}$, which complements and extends the corresponding results 
in~{\cite{Aleman:Cascante:Fabrega:Pascuas:Pelaez}:}
\begin{theorem}\label{thm1:boundedness:g-operators}
Let $\alpha\ge-1$, $0<p<\infty$, and $m,n\in\N$. Let $g\in\H(\D)$ and 
\begin{equation}
\label{eqn:thm1:boundedness:g-operators}
L=S^m_gT^n_g+\sum_{j=1}^mS^{m-j}_gT^{n_j}_gP_j(T_g),
\end{equation}
where  $P_j$ is a polynomial, $n_j\in\N$,
and $\frac{m-j}{n_j}\le\frac{m}n$, for $j=1,\dots,m$.
Then $L$ is bounded on  $A^p_{\alpha}$ if and only if
$g\in\B^s_{\alpha}$, where $s=\frac{m}n+1$.
\end{theorem}

Case $n=1$ in Theorem~{\ref{thm1:boundedness:g-operators}} is actually  covered 
by~{\cite[Theorem 1.2]{Aleman:Cascante:Fabrega:Pascuas:Pelaez}}. 
In this situation, the first term on the right,  $S_g^mT_g^n=\frac{1}{m+1}T_{g^{m+1}}$, is dominant: It is bounded on 
$A^p_\alpha$ if and only if $g\in\B^{m+1}_{\alpha}$ when $n=1$, while,  by Theorem~{\ref{thm:sharp:estimate:norm:word}}, the remaining  terms on the right are bounded under strictly weaker conditions than the above one. The situation is different when $n\ge 2$ and this illustrates the power of this new approach. For example, if 
 $n=2$ and $L=S^4_gT^2_g+S^2_gT_g$, then Theorem~{\ref{thm1:boundedness:g-operators}} shows   that 
 both terms in  the sum defining $L$,  $S^4_gT^2_g$ and $S^2_gT_g$, are bounded on  $A^p_\alpha$
if and only if $g\in \B^3_{\alpha}$,  {\em i.e.} there is no dominant term in the sense of the above observation. However, the same condition $g\in \B^3_{\alpha}$ is equivalent to the boundedness of $L$ on  $A^p_\alpha$.

The paper is organized as follows. Section~{\ref{section:symbols}} contains the basic properties of the symbol classes $\B^s_{\alpha}$. Section~{\ref{section:auxiliary:results}} gives the technical tools for the proofs: algebraic identities, Calder\'{o}n theorem, and   lower estimates of $g$-operator norms using iterate commutators.   As already pointed out  the proof of the main result, Theorem~{\ref{thm:sharp:estimate:norm:word}}, and the one of Theorem~{\ref{thm1:boundedness:g-operators}} are given in sections \ref{section:easycase}-\ref{section:proofs:of:thm1.3-1.4:proposition1.5}.
\vspace*{6pt}

\noindent
{\bf A word about  notation.} 
For a function $g:\D\to\mathbb{C}$ and $0<r<1$, we denote by $g_r$ its dilation, {\em i.e.} $g_r(z)=g(rz)$, $z\in\D$.
Throughout in what follows $\H_0(\D):=\{f\in \H(\D):\,f(0)=0\}$ and $A^p_{\alpha}(0):=A^p_{\alpha}\cap\H_0(\D)$. Moreover,
$\opnorm{L}_{\alpha,p}
:=\sup\{\|Lf\|_{\alpha,p}:f\in A^p_{\alpha}(0),\,\|f\|_{\alpha,p}\le1\}$.  
\vspace*{6pt}

Finally, we should like to emphasize that both theorems above continue to hold when $A^p_{\alpha}$ is replaced by $A^p_{\alpha}(0)$  and  $\|L\|_{\alpha,p}$ by $\opnorm{L}_{\alpha,p}$. This  will be effectively shown below.

\section{Classes of symbols}\label{section:symbols}

In this section we study the main properties of our symbol classes $BMOA^q$ and $\B^q$. We begin by giving alternative definitions of those classes based on Garsia type seminorms.  The equivalence between those definitions will allow us to prove the nesting property of the symbol classes (see Corollary \ref{cor1:prop:BMOAq:Bq:nested}), a crucial fact for the proof of our main results. 

In order to motivate the Garsia type definitions in the particular case of $BMOA$, observe that if $n\in\N$ and $g\in\H(\D)$, then the condition $g^n\in BMOA$ can be written in terms of Garsia's seminorms as
\begin{equation}\label{eqn:q:power:in:Bloch:Garsia:1}
	\sup_{a\in\D}\|g^n\circ\phi_a-g^n(a)\|_{H^2}<\infty.
\end{equation}
Since
$\|g^n\circ\phi_a-g^n(a)\|_{H^2}^2
=\|g\circ\phi_a\|_{H^{2n}}^{2n}-|g(a)|^{2n}$,
\eqref{eqn:q:power:in:Bloch:Garsia:1} is equivalent to
\begin{equation*}
	\sup_{a\in\D}\bigl(\|g\circ\phi_a\|_{H^{2n}}^{2n}-|g(a)|^{2n}\bigr)^{\frac1{2n}}<\infty,
\end{equation*}
which is a condition that makes sense even if $n$ is any positive real number.  
Those considerations  lead to the following definition.

\begin{definition}
	For $\alpha\ge-1$ and $0<q<\infty$, define 
\begin{equation*}
\boldopnorm{g}_{\alpha,q}:=\sup_{a\in\D}\bigl(\|g\circ\phi_a\|_{\alpha,2q}^{2q}-|g(a)|^{2q}\bigr)^{\frac1{2q}}
\qquad(g\in\H(\D)).
\end{equation*}
Recall that, by subharmonicity, $\|g\|_{\alpha,2q}^{2q}\ge|g(0)|^{2q}$, for every $g\in\H(\D)$, so the $\frac1{2q}$-power of the preceding definition makes sense. Observe that 
\begin{equation*}
\boldopnorm{g}_{\alpha,q}\le\boldopnorm{g}_{-1,q}
\qquad(g\in\H(\D),\,\alpha>-1,\,q>0).
\end{equation*} 
\end{definition}

\begin{definition}
For $\alpha\ge-1$ and $q>0$,  define 
\begin{equation*}
\boldnorm{g}_{\alpha,q}^{2q}:=    
\sup_{a\in\D}\int_{\D}
(1-|\phi_a|^2)^{\alpha+2}\,|\nabla|g|^q|^2\,dA
\qquad(g\in\H(\D)).
\end{equation*}

Since  $\|\cdot\|_{BMOA^q}=\boldnorm{\cdot}_{-1,q}$, for $q\ge1$, it is natural to extend the definition of $\|\cdot\|_{BMOA^q}$ to any $q>0$ by the identity 
\begin{equation}\label{eqn:BMOAqnorm:equal:boldnorm:-1}
\|\cdot\|_{BMOA^q}:=\boldnorm{\cdot}_{-1,q}\qquad(q>0).
\end{equation}

\end{definition}

\begin{remarks}$\mbox{ }$\label{remarks:symbols}
\begin{enumerate}[label={\sffamily{\arabic*)}},topsep=3pt, 
leftmargin=0pt, itemsep=4pt, wide, listparindent=0pt, itemindent=6pt] 

\item \label{remarks:symbols:1}
One might wonder why we consider $\|\cdot\|_{BMOA^q}$ and $\boldnorm{\cdot}_{\alpha,q}$, for any $q>0$, while $\|\cdot\|_{\B^q}$  only for  $q\ge1$ . The reason is quite simple: $\|g\|_{BMOA^q}<\infty$ and $\boldnorm{g}_{\alpha,q}<\infty$, for every $g\in\H(\overline{\D})$ and $q>0$, 
but if $0<q<1$ then there are functions $g\in\H(\overline{\D})$ such that $\|g\|_{\B^q}=\infty$. Indeed, it is clear that, if $q>0$ and either $g\in\H(\overline{\D})$ is zero free on $\overline{\D}$ or $g\equiv0$, then 
$|\nabla |g|^q|\in C(\overline{\D})$, and so both $\|g\|_{BMOA^q}$ and $\boldnorm{g}_{\alpha,q}$ are finite. On the other hand, if $q>0$ and $g\in\H(\overline{\D})$ has a zero of multiplicity $m\ge1$ at  $z_0\in\overline{\D}$, then 
$|\nabla |g|^q|(z)\simeq|z-z_0|^{qm-1}$ on a pointed neighborhood of $z_0$, so 
$|\nabla|g|^q|\in L^2(\D)$, and therefore $\|g\|_{BMOA^q}<\infty$ and $\boldnorm{g}_{\alpha,q}<\infty$. 
In particular, $\|g\|_{\B^q}=\infty$, for any $0<q<1$ and any function $g\in\H(\D)$ having at least one simple zero.

\item 
The classes $BMOA^q$ and  $\B^q$ are not vector spaces. Let us deal with the case $q=2$. Let  $g$ be a branch of the square root
of\[
h(z):=\log\biggl(\frac{e}{1-z}\biggr)\qquad(z\in\C\setminus[1,\infty)]).
\]
and let $B$ be a Blaschke product with positive zeros. Then it is clear that $g,B\in BMOA^2\subset \B^2$. However, $g+B$ does not belong to $\B^2$, and consequently neither it does not belong to $BMOA^2$. Indeed, it is clear that $g+B\in\B^2$ if and only if $gB\in\B$, but \cite[Theorem~1.6]{Girela:Gonzalez:Pelaez} ensures that $gB\not\in\B$.
\item 
The functionals  $\|\cdot\|_{\B^q}$,
$\boldopnorm{\,\cdot\,}_{\alpha,q}$ and $\boldnorm{\cdot}_{\alpha,q}$ are conformally invariant, {\em i.e.\ }  if $\|\cdot\|$ is any of these functionals, then $\|g\circ\phi\|=\|g\|$, for every $g\in\H(\D)$ and every conformal automorphism $\phi$ of $\D$.

\item \label{remarks:symbols:4}
The functionals  $\|\cdot\|_{\B^q}$,
$\boldopnorm{\,\cdot\,}_{\alpha,q}$ and $\boldnorm{\cdot}_{\alpha,q}$ are homogeneous and they vanish 
on constant functions.
Moreover, if $q>1$ then they  do not satisfy the triangle inequality.
Indeed,  if $\|\cdot\|$ is any of the preceding  functionals, for $q>1$,  $g(z)=z$, and $h_c\equiv c$, for $c>0$, then $\|g\|+\|h_c\|=\|g\|<\infty$, but
$\|g+h_c\|\to\infty$, as $c\to\infty$. 	
 
\item Let $\alpha\ge-1$ and let ${\bf M}^{+}(\D)$ be the set of all positive Borel measures on~{$\D$.}
Then we recall that the {\em $(\alpha+2)$-Carleson measure norm} of $\mu\in {\bf M}^{+}(\D)$ is defined by
\begin{equation*}
\|\mu\|_{\mathcal{C}(\alpha)}:=
\sup_{a\in\D}\frac{\mu(S(a))}{(1-|a|^2)^{\alpha+2}},
\end{equation*}
where $S(\varrho e^{i\theta}):=\{re^{it}:\varrho\le r<1,\,|t-\theta|\le\pi(1-\varrho)\}$, $0\le\varrho<1,\,\theta\in\R$.
Moreover, we have that
\begin{equation}\label{eqn:alpha+2-Carleson:measures:equiv}
\|\mu\|_{\mathcal{C}(\alpha)}
\simeq\sup_{\lambda\in\D}\int_{\D}{\bf B}_{\alpha}(z,\lambda)\,d\mu(z)
\simeq \sup_{\substack{f\in A^p_{\alpha}\\ \|f\|_{\alpha,p}=1}}
\|f\|^p_{L^p(\mu)}\quad(\mu\in {\bf M}^{+}(\D)),
\end{equation}
where 
\begin{equation*}
{\bf B}_{\alpha}(z,\lambda):=
\frac{(1-|\lambda|)^{\alpha+2}}{|1-\overline{\lambda}z|^{2\alpha+4}}\qquad(z\in\D,\,\lambda\in\overline{\D})
\end{equation*}
is the {\em Poisson-Bergman} (or {\em Berezin}) kernel (see \cite[Theorem II.3.9, Lemma VI.3.3]{Garnett}, for $\alpha=-1$, and \cite[Theorem 2.15]{Hedenmalm:Korenblum:Zhu}, for $\alpha=0$; the proof for $\alpha>-1$ is similar to the one for the case  $\alpha=0$).
Now observe that~{\eqref{eqn:alpha+2-Carleson:measures:equiv}} shows that
\begin{equation}
\label{eqn:derivative:symbol:Carleson:measure}
\boldnorm{g}_{\alpha,q}^{2q}\simeq\|\mu^{\alpha}_{g,q}\|_{\mathcal{C}(\alpha)}\simeq \sup_{\substack{f\in A^p_{\alpha}\\ \|f\|_{\alpha,p}=1}}
\|f\|^p_{L^p(\mu^{\alpha}_{g,q})}
\qquad(g\in\H(\D),\,q>0),   
\end{equation}
where 
\begin{equation}\label{eqn:Carleson:measure:Balphaq}
d\mu^{\alpha}_{g,q}(z):=
(1-|z|^2)^{\alpha+2}|\nabla|g|^q(z)|^2\,dA(z).
\end{equation}
\end{enumerate}
\end{remarks}
\vspace*{6pt}

We are going to state and prove several  properties of $BMOA^q$ and $\B^q$. 
First we give an alternative description of $\|\cdot\|_{\B^q}$ without using derivatives, which shows the growth of the functions in $\B^q$.
\begin{proposition}\label{prop:growth:Bq}
Let $q\ge1$. Then
\begin{equation}\label{eqn:Bq:norm:equals:Lipschitz:norm}
\|g\|^q_{\B^q}=\sup_{\substack{w,z\in\D\\w\ne z}}
\frac{\bigl||g(w)|^q-|g(z)|^q\bigr|}{\beta(w,z)}
\qquad(g\in\H(\D)),
\end{equation}
where $\beta(w,z):=\frac12\log\frac{1+|\phi_w(z)|}{1-|\phi_w(z)|}$ is the hyperbolic distance in $\D$.
\end{proposition}

\begin{proof}
Let $g\in\H(\D)$ and let $C_q(g)$ be the
right hand side term in~{\eqref{eqn:Bq:norm:equals:Lipschitz:norm}}. 
Since 
${\displaystyle\lim_{w\to z}|w-z|/\beta(w,z)=1-|z|^2}$,
 for any $z\in\D$, we have that
\begin{equation*}
(1-|z|^2)|\nabla|g|^q|(z)
=\limsup_{w\to z}
\frac{\bigl||g(w)|^q-|g(z)|^q\bigr|}{\beta(w,z)}
\le C_q(g)
\qquad(z\in\D).
\end{equation*}
It follows that $\|g\|^q_{B^q}\le C_q(g)$.
On the other hand, if $z\in\D$ then
\begin{align*}
\bigl||g(z)|^q-|g(0)|^q\bigr|
&\le\int_0^1\left|\frac{d\,}{dt}|g(tz)|^q\right|\,dt=|z|\int_0^1|\nabla|g|^q(tz)|\,dt 
\\
\nonumber
&
\le\|g\|^q_{\B^q}\int_0^1\frac{|z|\,dt}{1-t^2|z|^2}
=\|g\|^q_{\B^q}\,\beta(z,0).
\end{align*}
By replacing in the preceding inequality $g$ and $z$ by $g\circ\phi_w$ and $\phi_w(z)$, respectively,  and applying the conformal invariance of $\|\cdot\|_{\B^q}$, we obtain that 
$\bigl||g(w)|^q-|g(z)|^q\bigr|\le\|g\|^q_{\B^q}\,\beta(w,z)$,  
and therefore $C_q(g)\le\|g\|^q_{\B^q}$. 
\end{proof} 
Note that $\|\cdot\|_{\B^1}=\|\cdot\|_{\B}$, so the identity \eqref{eqn:Bq:norm:equals:Lipschitz:norm} for $q=1$ gives an expression of the 
Bloch seminorm $\|\cdot\|_{\B}$ which seems to be new. It should be compared 
with the following well known formula (see \cite[Theorem 5.5]{Zhu}):
\begin{equation*}
\|g\|_{\B}=\sup_{\substack{w,z\in\D\\w\ne z}}
\frac{\bigl| g(w)-g(z)\bigr|}{\beta(w,z)}
\qquad(g\in\H(\D)).
\end{equation*}

Now we prove that if $q\ge1$ then $\boldnorm{\cdot}_{\alpha,q}$ and $\boldopnorm{\,\cdot\,}_{\alpha,q}$ are equivalent, for $\alpha\ge-1$,  
and also $\|\cdot\|_{\B^q}$ and $\boldopnorm{\,\cdot\,}_{\alpha,q}$ are equivalent, for $\alpha>-1$.
\begin{proposition}
\label{prop:equiv:Garsia:Carleson:symbols}
 $\mbox{}$
For every $\alpha\ge-1$ there are two constants 
 $0<c_{\alpha}<1$ and $C_{\alpha}>1$ such that
\begin{gather}
\label{eqn:prop:equiv:Garsia:Carleson:symbols:1}
c_{\alpha}\boldnorm{g}_{\alpha,q}^{2q}
\le\boldopnorm{g}_{\alpha,q}^{2q}
\qquad(g\in\H(\D),\,q>0).
\\
\label{eqn:prop:equiv:Garsia:Carleson:symbols:2}
\boldopnorm{g}_{\alpha,q}^{2q}\le C_{\alpha}\boldnorm{g}_{\alpha,q}^{2q}
\qquad(g\in\H(\D),\,q\ge1).
\end{gather}
 Moreover, for every $\alpha>-1$  there are constants  $0<c_{\alpha}'<1$ and $C'_{\alpha}>1$ such that
\begin{equation}
\label{eqn:prop:equiv:Garsia:Carleson:symbols:3}
c_{\alpha}'\|g\|_{\B^q}^{2q}\le\boldopnorm{g}_{\alpha,q}^{2q}
\le C'_{\alpha}\|g\|_{\B^q}^{2q}\qquad(g\in\H(\D),\,q\ge1).
\end{equation}
\end{proposition}

In order to prove Proposition \ref{prop:equiv:Garsia:Carleson:symbols}
we need the following two lemmas.

\begin{lemma}\label{lem:subharmonic1}
If $g\in\H(\D)$ and $q\ge1$ then $v=|\nabla|g|^q|^2$ is a nonnegative subharmonic function on $\D$.
\end{lemma}

\begin{proof}
If $q=1$ then $v=|g'|^2$ is clearly a nonnegative subharmonic function on $\D$. 
If $q>1$ then $u=(2q-2)\log|g|+2\log|g'|$ is a subharmonic function on $\D$, and so $v=q^2e^u$ is a nonnegative subharmonic function on $\D$.
\end{proof}

\begin{lemma}\label{lem:subharmonic2}
Let $w$ be a positive radial integrable function on $\D$. Then there is a constant $C>0$ such that
\begin{equation*}
\int_{\D}v(z)w(z)\,dA(z)\le C\int_{\frac12<|z|<1}v(z)w(z)\,dA(z),	
\end{equation*}
for every  nonnegative subharmonic function  $v$ on $\D$.
\end{lemma}

\begin{proof}
Let $v$ be  a nonnegative subharmonic function on $\D$.	First note that
\begin{equation*}
	\int_{\D}v(z)w(z)\,dA(z)= 
	\biggl\{\int_0^{\frac12}+\int_{\frac12}^1\biggr\} w(r)\lambda(r)\,2r\,dr=I_0+I_1
\end{equation*}
	where $\lambda(r):=\int_{\T}v_r\,d\sigma$, for $0<r<1$.
	We want to prove that $I_0\lesssim I_1$.
	Since $\lambda$ is an increasing function, we have that
	$\lambda(r)\le\lambda(\tfrac12)$, for $0<r<\frac12$,
	and $\lambda(\tfrac12)\le\lambda(r)$, for $\frac12<r<1$.
	By integrating these inequalities against $w(r)\,2r\,dr$ along the intervals $(0,\frac12)$ and $(\frac12,1)$, we get that 
\begin{equation*}
I_0\le \lambda(\tfrac12)\int_{|z|<\frac12}w(z)\,dA(z)
\quad\mbox{and}\quad
\lambda(\tfrac12)\int_{\frac12<|z|<1}w(z)\,dA(z)
\le I_1.
\end{equation*}
Therefore $I_0\lesssim I_1$, and that ends the proof.	
\end{proof}

\begin{proof}[{\bf Proof of Proposition \ref{prop:equiv:Garsia:Carleson:symbols}}]	Let $q>0$ and $g\in\H(\D)$. By making $r\to1^{-}$ in the Hardy-Stein identity (see, for instance, \cite[Theorem 2.18]{Pavlovic})
\begin{equation*} \label{eqn:Hardy-Stein:identity}
\|g_r\|_{H^{2q}}^{2q}-|g(0)|^{2q}=
\int_{D(0,r)}|\nabla|g|^q(z)|^2\,\log\frac{r^2}{|z|^2}\,dA(z)
\quad(0<r<1),
\end{equation*}
 the monotone convergence theorem gives that
\begin{equation}\label{eqn:prop:equiv:Garsia:Carleson:symbols:31}
	\|g\|^{2q}_{H^{2q}}-|g(0)|^{2q}
	=\int_{\D}|\nabla|g|^q(z)|^2\,\log\frac1{|z|^2}\,dA(z).
\end{equation}
Moreover, if $\alpha>-1$, we may integrate~{\eqref{eqn:Hardy-Stein:identity}}
against $(\alpha+1)(1-r^2)^{\alpha}\,2r\,dr$ along the interval $(0,1)$ and use Tonelli's theorem to  obtain that
\begin{equation}\label{eqn:prop:equiv:Garsia:Carleson:symbols:4}
	\|g\|^{2q}_{\alpha,2q}-|g(0)|^{2q}=\int_{\D}|\nabla|g|^q(z)|^2w_{\alpha}(z)\,dA(z),
\end{equation}
where 
\[
w_{\alpha}(z)
:=(\alpha+1)
\int_{|z|}^1(1-r^2)^{\alpha}\,2r\,\log\frac{r^2}{|z|^2}\,dr
=(\alpha+1)
\int_{|z|^2}^1(1-r)^{\alpha}\,\log\frac{r}{|z|^2}\,dr.
\]
Now we estimate $w_{\alpha}$ as follows:
\begin{equation}\label{eqn:prop:equiv:Garsia:Carleson:symbols:5}
c_{\alpha}(1-|z|^2)^{\alpha+1}\log\frac1{|z|^2}\le w_{\alpha}(z)\le (1-|z|^2)^{\alpha+1}\log\frac1{|z|^2}\qquad(z\in\D),
\end{equation}
where $0<c_{\alpha}<1$ is a constant. First note that 
\begin{equation}\label{eqn:prop:equiv:Garsia:Carleson:symbols:6}
w_{\alpha}(z)
=(1-|z|^2)^{\alpha+1}\log\frac1{|z|^2}
-(\alpha+1)
\int_{|z|^2}^1(1-r)^{\alpha}\,\log\frac1r\,dr,	
\end{equation}
which shows the right hand estimate in \eqref{eqn:prop:equiv:Garsia:Carleson:symbols:5}, since the last integral
in \eqref{eqn:prop:equiv:Garsia:Carleson:symbols:6} is positive.
The left hand estimate in \eqref{eqn:prop:equiv:Garsia:Carleson:symbols:5} is a direct consequence of the inequality $\sup_{0<x<1}\psi_{\alpha}(x)<1$,
where
\[
\psi_{\alpha}(x)=\frac{(\alpha+1)
	\int_{x^2}^1(1-r)^{\alpha}\,\log\frac1r\,dr}{(1-x^2)^{\alpha+1}\log\frac1{x^2}}
\qquad(0<x<1).
\]
The above inequality follows from the continuity of $\psi_{\alpha}$ 
on the interval $(0,1)$ together with the facts $\psi_{\alpha}(x)<1$, for every
$x\in(0,1)$, $\lim_{x\to0^{+}}\psi_{\alpha}(x)=0$, and
$\lim_{x\to1^{-}}\psi_{\alpha}(x)=\frac{\alpha+1}{\alpha+2}<1$.

Next observe that \eqref{eqn:prop:equiv:Garsia:Carleson:symbols:31} shows that \eqref{eqn:prop:equiv:Garsia:Carleson:symbols:4} holds for $\alpha=-1$ with $w_{-1}(z)=\log\frac1{|z|^2}$, which clearly satisfies the estimate \eqref{eqn:prop:equiv:Garsia:Carleson:symbols:5}.
Thus \eqref{eqn:prop:equiv:Garsia:Carleson:symbols:4} and \eqref{eqn:prop:equiv:Garsia:Carleson:symbols:5} hold for any $\alpha\ge-1$. Now taking into account  \eqref{eqn:prop:equiv:Garsia:Carleson:symbols:4}, \eqref{eqn:prop:equiv:Garsia:Carleson:symbols:5}, Lemmas \ref{lem:subharmonic1} and \ref{lem:subharmonic2},  and the inequalities $1-|z|^2\le\log\frac1{|z|^2}$, for any $z\in\D$, and
$\log\frac1{|z|^2}\le\frac{1-|z|^2}{|z|^2}\le4(1-|z|^2)$, for $\frac12<|z|<1$, we obtain  the  estimates 
\begin{gather}
\label{eqn:prop:equiv:Garsia:Carleson:symbols:7}	
c_{\alpha}\int_{\D}(1-|z|^2)^{\alpha+2}v_{g,q}(z)\,dA(z)\le\|g\|^{2q}_{\alpha,2q}-|g(0)|^{2q}\quad(q>0)
\\
\label{eqn:prop:equiv:Garsia:Carleson:symbols:8}
\|g\|^{2q}_{\alpha,2q}-|g(0)|^{2q}
\le C_{\alpha}\int_{\D}(1-|z|^2)^{\alpha+2}v_{g,q}(z)\,dA(z)
\quad(q\ge1),
\end{gather}
where  $v_{g,q}=|\nabla|g|^q|^2$ and $C_{\alpha}>0$ is a constant.
By replacing $g$ by $g\circ\phi_a$ in~{\eqref{eqn:prop:equiv:Garsia:Carleson:symbols:7}} and~{\eqref{eqn:prop:equiv:Garsia:Carleson:symbols:8}},  
we deduce the estimates \eqref{eqn:prop:equiv:Garsia:Carleson:symbols:1} and \eqref{eqn:prop:equiv:Garsia:Carleson:symbols:2}.

Finally, by taking into account~{\eqref{eqn:prop:equiv:Garsia:Carleson:symbols:7}}, \eqref{eqn:prop:equiv:Garsia:Carleson:symbols:8} and the fact that $v_{g,q}$ is subharmonic, for $q\ge1$,
we get
\begin{gather}
\label{eqn:prop:equiv:Garsia:Carleson:symbols:9}
c_{\alpha}'|\nabla|g|^q(0)|^2=
c_{\alpha}v_{g,q}(0)\int_0^1(1-r^2)^{\alpha+2}\,dr\le
\|g\|^{2q}_{\alpha,2q}-|g(0)|^{2q}
\\
\label{eqn:prop:equiv:Garsia:Carleson:symbols:10}
\|g\|^{2q}_{\alpha,2q}-|g(0)|^{2q}\le C_{\alpha}\|g\|_{\B^q}^{2q}\int_{\D}(1-|z|^2)^{\alpha}\,dA(z)
=C_{\alpha}'\|g\|_{\B^q}^{2q}.
\end{gather}
Then replace $g$ by $g\circ\phi_z$, $z\in\D$, in \eqref{eqn:prop:equiv:Garsia:Carleson:symbols:9} and \eqref{eqn:prop:equiv:Garsia:Carleson:symbols:10}, and take into account the conformal invariance of $\|\cdot\|_{\B^q}$ to deduce~{\eqref{eqn:prop:equiv:Garsia:Carleson:symbols:3}}, which ends the proof.
\end{proof}

\begin{corollary}
\label{cor:prop:equiv:Garsia:Carleson:symbols}
Let $\alpha>-1$ and $q\ge1$. Then:
\begin{align}
\label{eqn:equiv:norms:Bq}
\|g\|_{\B^q}\simeq\boldnorm{g}_{\alpha,q}&\simeq\boldopnorm{g}_{\alpha,q}
\qquad(g\in\H(\D)).
\\
\label{eqn:equiv:norms:BMOAq}
\|g\|_{BMOA^q}&\simeq\boldopnorm{g}_{-1,q}
\qquad(g\in\H(\D)).
\\
\label{eqn:norm:Bq:lesssim:norm:BMOAq}
\|g\|_{\B^q}&\lesssim\|g\|_{BMOA^q}
\qquad(g\in\H(\D)).
\end{align}
Moreover,  $BMOA^q\varsubsetneq\B^q$.
\end{corollary}
\begin{proof}
All the estimates directly follow from Proposition \ref{prop:equiv:Garsia:Carleson:symbols}. Then the inclusion is a direct consequence of \eqref{eqn:norm:Bq:lesssim:norm:BMOAq}. Note that this inclusion is proper since, by~\cite{Campbell:Cima:Stephenson}, there exists a zero free function $h\in\B\setminus BMOA$, and so any branch of the $\frac1q$-power of $h$ belongs to 
$\B^q\setminus BMOA^q$. 
\end{proof}

Next result shows that a function $g\in\H(\D)$ belongs to $\B^q_{\alpha}$ if and only if  $\sup_{0<r<1}\boldnorm{g_r}_{\alpha,q}<\infty$, 
where $g_r(z):=g(rz)$. Note that $g_r\in\H(\overline{\D})$.

\begin{proposition}\label{prop:radialized:symbol}
	For every $\alpha\ge-1$ we have:
\begin{gather}
\label{eqn:prop:radialized:symbol:1}
\boldnorm{g}_{\alpha,q}
\le\liminf_{r\to 1}\boldnorm{g_r}_{\alpha,q}
\qquad(g\in\H(\D),\,q>0).
\\
\label{eqn:prop:radialized:symbol:2}
\sup_{0<r<1}\boldnorm{g_r}_{\alpha,q}\lesssim \boldnorm{g}_{\alpha,q}
\qquad(g\in\H(\D),\,q\ge1).
\end{gather}
\end{proposition}

\begin{proof}	
For any $a\in\D$,  Fatou's lemma shows that
\begin{equation*}
\int_{\D}(1-|\phi_a|^2)^{\alpha+2} |\nabla|g|^q|^2\,dA
\le \liminf_{r\to 1^-} \int_{\D}(1-|\phi_a|^2)^{\alpha+2} |\nabla|g_r|^q|^2\,dA,
\end{equation*}
which gives \eqref{eqn:prop:radialized:symbol:1}.
In order to show \eqref{eqn:prop:radialized:symbol:2}, by \eqref{eqn:derivative:symbol:Carleson:measure} we only have to prove that
\begin{equation}\label{eqn:prop:radialized:symbol:3}	
\frac{\mu_{g_r,q}^{\alpha}(S(a))}{(1-|a|^2)^{\alpha+2}}
\lesssim \boldnorm{g}_{\alpha,q}^{2q} 
\qquad(g\in\H(\D),\,a\in\D,\,0<r<1,\,q\ge1).
\end{equation}
	Note that
	\begin{equation*}
		\frac{1-|z|^2}{1-|a|^2}\lesssim 1-|\phi_a(z)|^2
		\qquad(a\in\D,\,z\in S(a)),
	\end{equation*}
	and so
\begin{align*}
\frac{\mu_{g_r,q}^{\alpha}(S(a))}{(1-|a|^2)^{\alpha+2}}
&\lesssim 
\int_{S(a)}(1-|\phi_a(z)|^2)^{\alpha+2}\, |\nabla|g|^q(rz)|^2\,r^2\,dA(z)
\\ 
&\le \int_{\D}(1-|\phi_a(z)|^2)^{\alpha+2}\, |\nabla|g|^q(rz)|^2\,r^2\,dA(z)
\\ 
&= \int_{r\D}\left(1-|\phi_a\left(\tfrac{w}{r}\right)|^2\right)^{\alpha+2} |\nabla|g|^q(w)|^2\,dA(w).
\end{align*}
	Now 
	\begin{equation*}
		1-|\phi_a\left(\tfrac{w}{r}\right)|^2
		=\frac{(1-|a|^2)(r^2-|w|^2)}{|r-\overline{a}w|^2}
		\le\frac{(1-|a|^2)(1-|w|^2)}{|r-\overline{a}w|^2}.
	\end{equation*}
	But $|1-\overline{a}w|\le 1-r+|r-\overline{a}w|
	\le2|r-\overline{a}w|$,
	for any $|w|<r$ and $1-|a|\ge\frac{1-r}r$, and so
	\begin{equation*}
		1-|\phi_a\left(\tfrac{w}{r}\right)|^2
		\le 4(1-|\phi_a(w)|^2)
		\qquad(|w|<r,\,\tfrac{1-r}r\le1-|a|<1).
	\end{equation*}
	Therefore
	\begin{equation*}
	\frac{\mu_{g_r,q}^{\alpha}(S(a))}{(1-|a|^2)^{\alpha+2}}\lesssim \boldnorm{g}^{2q}_{\alpha,q} 
		\quad(g\in\H(\D),\,0<r<1,\,\tfrac{1-r}r\le1-|a|<1,\,q>0).
	\end{equation*}
Now assume that $a\in\D$, $0<r<1$, and $1-|a|<\tfrac{1-r}r$. Since, by 
Lemma~{\ref{lem:subharmonic1},} $v=|\nabla|g|^q|^2$ is a subharmonic function on $\D$, we have that 
\begin{align*}
\mu_{g_r,q}^{\alpha}(S(a))
&=r^2\int_{S(a)} (1-|z|^2)^{\alpha+2}\,v(rz)\,dA(z)
\\
&\le r^2\biggl(\sup_{|z|=r}v(z)\biggr)\int_{S(a)} (1-|z|^2)^{\alpha+2}\,dA(z)
\\
&\lesssim r^2(1-|a|^2)^{\alpha+4}\sup_{|z|=r}v(z)
\\
&\lesssim (1-|a|^2)^{\alpha+2}\sup_{|z|=r}v(z)(1-r^2)^2
\\
&\le(1-|a|^2)^{\alpha+2}\|g\|_{\B^q}^{2q} 
\lesssim(1-|a|^2)^{\alpha+2}\boldnorm{g}_{\alpha,q}^{2q},
\end{align*}
where the last estimate follows from Corollary~{\ref{cor:prop:equiv:Garsia:Carleson:symbols}}, provided that $q\ge1$.
And that finishes the proof of~{\eqref{eqn:prop:radialized:symbol:3}}.
\end{proof}

Finally, we show that the classes $BMOA^q$ and $\B^q$ are nested:
\begin{proposition}\label{prop:BMOAq:Bq:nested}
Let $\alpha\ge-1$ and $0<q<r$. 
Then $\boldopnorm{g}_{\alpha,q}\le \boldopnorm{g}_{\alpha,r}$,
for any $g\in\H(\D)$. 
\end{proposition} 
\noindent
The proof of Proposition~{\ref{prop:BMOAq:Bq:nested}} reduces to the following lemma.
\begin{lemma}
Let $\alpha\ge-1$, $0<q<r$, and  $g\in\H(\D)$. Then:
\begin{equation}
\label{eqn:lemma:prop:BMOAq:Bq:nested:ineq0}
\bigl(\|g\|_{\alpha,2q}^{2q}-|g(0)|^{2q}\bigr)^{\frac1{2q}}
\le \bigl(\|g\|_{\alpha,2r}^{2r}-|g(0)|^{2r}\bigr)^{\frac1{2r}}.
\end{equation}
\end{lemma}
\begin{proof}
First note that
$\|g\|_{\alpha,2r}\ge\|g\|_{\alpha,2q}$, and so \eqref{eqn:lemma:prop:BMOAq:Bq:nested:ineq0}  holds in the case that $g(0)=0$. 	
The general case follows from the inequality 
$\|g\|_{\alpha,2r}\ge\|g\|_{\alpha,2q}$ and a simple argument. Indeed,  
\[
\|g\|^{2r}_{\alpha,2r}-|g(0)|^{2r}
\ge\|g\|^{2r}_{\alpha,2q}-|g(0)|^{2r}
\ge\bigl(\|g\|^{2q}_{\alpha,2q}-|g(0)|^{2q}\bigr)^{\frac{r}q},
\]
where the last inequality 
 is a consequence of the classical superadditivity inequality
\[
(x+y)^s\ge x^s+y^s\qquad(x,y\ge0,\,s\ge1).
\] 
(Recall that any convex function $\varphi:[0,\infty)\to\R$
whith $\varphi(0)=0$ is superadditive,
 {\em i.e.} $\varphi(x+y)\ge\varphi(x)+\varphi(y)$,
 for any $x,y\ge0$.)\newline
Hence 
\[
\bigl(\|g\|^{2r}_{\alpha,2r}-|g(0)|^{2r}\bigr)^{\frac1{2r}}
\ge\bigl(\|g\|^{2q}_{\alpha,2q}-|g(0)|^{2q}\bigr)^{\frac1{2q}}
\]
and the proof is complete.
\end{proof}
\begin{corollary}\label{cor1:prop:BMOAq:Bq:nested}
If $1\le q<r$ and $\alpha\ge-1$ then 
\begin{equation}\label{eqn:BMOAq:Bq:nested:norms} 
\|g\|_{\B^q_{\alpha}}\lesssim\|g\|_{\B^r_{\alpha}}\qquad(g\in\H(\D)).
\end{equation}
Moreover, $BMOA^r\varsubsetneq BMOA^q$ and $\B^r\varsubsetneq\B^q$. 
\end{corollary}
\begin{proof}
The estimate directly follows from  Proposition~{\ref{prop:BMOAq:Bq:nested}} 
and estimates~{\eqref{eqn:equiv:norms:Bq}} and~{\eqref{eqn:equiv:norms:BMOAq}}. Hence the inclusions hold, so we only have to prove that they are proper. 
Let $\log$ be the principal branch of the logarithm on 
$\C\setminus(-\infty,0]$, and consider the holomorphic function
\[
h(z):=\log\biggl(\frac{e}{1-z}\biggr)\qquad(z\in\C\setminus[1,\infty)]).
\]
It is clear that $h$ is a zero-free function in $BMOA\subset\B$. 
Let $g=h^{1/q}$ be any branch of the $1/q$-power of $h$ on $\D$, for $q\ge1$. Then $g\in BMOA^q\subset\B^q$. 
Let $r>q$.	Since, by Proposition~{\ref{prop:growth:Bq}}, any function $f\in\B^r$ satisfies
	\[
	|f(z)|^r\lesssim\log\biggl(\frac{e}{1-|z|}\biggr)\qquad(z\in\D),
	\]
	and 
	\[
	\lim_{\rho\to1^{-}}\frac{|g(\rho)|^r}{\log\bigl(\frac{e}{1-\rho}\bigr)}
	=\lim_{\rho\to1^{-}}\frac{|h(\rho)|^{r/q}}{\log\bigl(\frac{e}{1-\rho}\bigr)}
	=\lim_{\rho\to1^{-}}\biggl[\log\biggl(\frac{e}{1-\rho}\biggr)\biggr]^{r/q-1}=\infty,
	\]
 we deduce that $g\notin\B^r$, and so
	$g\notin BMOA^r$, which ends the proof.
\end{proof}
\begin{corollary}
If $\alpha\ge-1$, then there exist two constants $0<c_{\alpha}<1$ and $C_{\alpha}>1$ satisfying that
\begin{equation}
	\label{eqn:key:estimate:Carleson:norms:symbols0}	
c_{\alpha}^{1/ q}\boldnorm{g}_{\alpha,q}
\le C_{\alpha}^{1/r}\boldnorm{g}_{\alpha,r}
\qquad(g\in\H(\D),\,0<q<r,\,r\ge1).
\end{equation}
Moreover, there are two constants $0<c<1$ and $C>1$ such that	
\begin{align}
\label{eqn:key:estimate:Carleson:norms:symbols1}
c^{1/ q}\|g\|_{BMOA^q}
&\le C^{1/r}\|g\|_{BMOA^r}
  \quad(g\in\H(\D),\,0<q<r,\,r\ge1).
\\
\label{eqn:key:estimate:Carleson:norms:symbols2}
c^{1/ q}\|g\|_{\B^q}
&\le C^{1/r}\|g\|_{\B^r}
	\qquad(g\in\H(\D),\,1\le q<r).
\end{align}	
\end{corollary}
\begin{proof}
It is a direct consequence of
Propositions~{\ref{prop:equiv:Garsia:Carleson:symbols}} 
and~{\ref{prop:BMOAq:Bq:nested}.} 
\end{proof}

\section{Auxiliary results}\label{section:auxiliary:results}

In this section we collect auxiliary tools which will be employed to prove our boundedness results.

\subsection{Algebraic decompositions of $g$-words}
The main result of this subsection is the following algebraic decomposition theorem of $g$-words, which, as we will see, will reduce the proof of Theorem \ref{thm:sharp:estimate:norm:word} to obtain precise norm estimates of the operators $S^k_gT^j_g$.
In the statement of the next result we denote by $\Pi_0:\H(\D)\to\H_0(\D)$ the operator given by $\Pi_0f=f-f(0)=f_0$.

\begin{theorem}\label{thm:global:decomposition}
Let $L\in W_g(\ell,m,n)$, where $\ell,m,n\in\N_0$, $m+n\ge1$, and  $k=\ell+m\ge1$. Then there exist integers $a_j,b_j$, $j=1,\dots,k$, which do not dependent on $g$ and satisfy
\begin{align*}
L&= (1-\delta_L) S^k_gT^n_g+\delta_LS^k_gT_g^n\Pi_0 \\
\nonumber
 &\quad +\sum_{j=1}^ka_j\,S_g^{k-j}T_g^{n+j}
        +\sum_{j=1}^kb_j\,S_g^{k-j}T_g^{n+j}\Pi_0,
\end{align*}	
where $\delta_L=0$, if $L$ ends in $T_gM^i_g$, for some $i\in\N_0$, and  $\delta_L=1$, if $L$ ends in $S_gM^i_g$, for some $i\in\N_0$. In particular,
\begin{equation}
\label{eqn:global:decomposition:H0}
L= S^k_gT^n_g+\sum_{j=1}^kc_j\,S_g^{k-j}T_g^{n+j}
\quad\mbox{on $\H_0(\D)$,}	
\end{equation}
where the $c_j$'s are integers independent of $g$.
\end{theorem} 

The proof of Theorem \ref{thm:global:decomposition} will be splitted into 
several propositions. In order to state those propositions it will be useful to introduce the following notation.

\begin{definition}
For $m,n\in\N_0$, $m+n\ge1$,	
let $W_g(m,n)=W_g(0,m,n)$ and, for $m\in\N_0$ and $n\in\N$, let
$W^T_g(m,n)$ be the set of all $g$-words $L\in W_g(m,n)$ ending in $T_g$.
Similarly, for $m\in\N$ and $n\in\N_0$, $W^S_g(m,n)$ denotes the set of all $g$-words $L\in W_g(m,n)$ ending in $S_g$.
\end{definition}

The first proposition reduces the proof of Theorem \ref{thm:global:decomposition} to the case where $L\in W_g(m,n)$.

\begin{proposition}\label{prop:delete:M:in:a:g-word}
	Let $L\in W_g(\ell,m,n)$, where $\ell\in\N$ and $m,n\in\N_0$ satisfy $m+n\ge1$. 
	\begin{enumerate}[label={\sf\alph*)},topsep=3pt, 
		leftmargin=0pt, itemsep=4pt, wide, listparindent=0pt, itemindent=6pt] 
		
		\item If $L$ does not end in $M_g$, then 
		\begin{equation}\label{eqn:g-words:not:ending:in:Mg}
			L=\LL_0+\dots+\LL_\ell,
		\end{equation}
		where  $\LL_0\in W_g(m+\ell,n)$ and $\LL_j$ is a sum of $\binom{\ell}{j}$ not necessarily different $g$-words in  $W_g(m+\ell-j,n+j)$, for $j=1,\dots,\ell$.
		Moreover,  $\LL_0\in W^T_g(m+\ell,n)$ if  $L$ ends in $T_g$, and  $\LL_0\in W^S_g(m+\ell,n)$, if
		$L$ ends in $S_g$.
		
		\item If $L$ ends in $T_gM_g^k$, with $1\le k\le\ell$, then 
		\begin{equation}\label{eqn:g-words:ending:in:TgMg^k}
			L=\LL_0+\dots+\LL_{\ell-k},
		\end{equation}
		where\,  $\LL_0\in W^T_g(m+\ell,n)$  and, in the case that $k<\ell$, $\LL_j$ is a sum of $\binom{\ell-k}{j}$ not necessarily different $g$-words in $W_g(m+\ell-j,n+j)$, for $j=1,\dots,\ell-k$.
		
		\item If $L$ ends in $S_gM_g^k$, with $1\le k\le\ell$, then 
		\begin{equation}\label{eqn:g-words:ending:in:SgMg^k}
			L=\LL_0+\cdots+\LL_{\ell-k}+k\LL_{\ell-k+1},
		\end{equation}
		where $\LL_0\in W^S_g(m+\ell,n)$,
		$\LL_{\ell-k+1}\in W_g(m+k-1,n+\ell-k+1)$,
		and, in the case that $k<\ell$,
		$\LL_j$ is a sum of $k\binom{\ell-k}{j-1}+\binom{\ell-k}{j}$ not necessarily different $g$-words in $W_g(m+\ell-j,n+j)$, for $j=1,\dots,\ell-k$.
	\end{enumerate}
\end{proposition}

\begin{proof} $\mbox{}$ 
	
	\begin{enumerate}[label={\sffamily{\alph*)}},topsep=3pt, 
		leftmargin=0pt, itemsep=4pt, wide, listparindent=0pt, itemindent=6pt] 
		
		\item 
Since $L=L_1\dots L_{\ell+m+n}$, with $L_{\ell+m+n}\ne M_g$,  identity \eqref{eqn:formula} allow us to replace any $L_s=M_g$, $1\le s<\ell+ m+n$, by $S_g+T_g$, and therefore an elementary combinatorial argument shows that~{\eqref{eqn:g-words:not:ending:in:Mg}} holds, where  $\LL_0\in W^T_g(m+\ell,n)$, if  $L$ ends in $T_g$, and  $\LL_0\in W^S_g(m+\ell,n)$, if
		$L$ ends in $S_g$.
		
		\item In this case, $L=\widetilde{L}T_gM_g^k$, where
		$\widetilde{L}\in W_g(\ell-k,m,n-1)$. Since 
		\begin{equation}\label{eqn:ST-decomposition:TgMg^k}
			T_g M_g^k=T_gM_{g^k}=S_g^kT_g,
		\end{equation}
		we have that 
		$L=\widetilde{L}S_g^kT_g\in W_g(\ell-k,m+k,n)$. It follows that we may apply part~{\sf a)} to deduce that~{\eqref{eqn:g-words:ending:in:TgMg^k}} holds, where $\LL_0\in W^T_g(m+\ell,n)$.
		
		\item Then  
		$L=\widetilde{L}S_gM_g^k$, with 
		$\widetilde{L}\in W_g(\ell-k,m-1,n)$. To study this case note that
		\begin{equation}\label{eqn:ST-decomposition:SgMg^k}
			S_gM_g^k=kS_g^kT_g+S_g^{k+1}.
		\end{equation}
		Indeed, if $f\in\H(\D)$ then 
		\begin{equation*}
			(S_gM_g^kf)'=g(M_g^kf)'=kg'g^kf+g^{k+1}f'=k(T_gM_g^kf)'+(S_g^{k+1}f)',
		\end{equation*} 
		so $S_gM_g^kf=k\,T_gM_g^kf+S_g^{k+1}f=kS_g^kT_gf+S_g^{k+1}f$,
		where the last identity follows from~\eqref{eqn:ST-decomposition:TgMg^k}.  
		Therefore~{\eqref{eqn:ST-decomposition:SgMg^k}}  holds, and, as a consequence, 
		we have that $L=k\widetilde{L}S_g^kT_g+\widetilde{L}S_g^{k+1}$.
		Since $\widetilde{L}S^k_gT_g\in W_g(\ell-k,m+k-1,n+1)$ and 
		$\widetilde{L}S_g^{k+1}\in W_g(\ell-k,m+k,n)$, we may apply part~{\sf a)} to get that
		\begin{equation*}
			L= k\sum_{j=0}^{\ell-k}\LL_{1,j}+\sum_{j=0}^{\ell-k}\LL_{2,j}
		\end{equation*}
		where $\LL_{1,j}$ is a sum of $\binom{\ell-k}{j}$ not necessarily different $g$-words belonging to $W_g(m+\ell-j-1,n+j+1)$ and $\LL_{2,j}$ is a sum of $\binom{\ell-k}{j}$ not necessarily different $g$-words in $W_g(m+\ell-j,n+j)$,
		where  $\LL_{2,0}\in W^S_g(m+\ell,n)$. Then, for $j=1,\dots,\ell-k+1$,
		$\widehat{\LL}_{1,j}:=\LL_{1,j-1}$ is a sum of $\binom{\ell-k}{j-1}$ not necessarily different $g$-words belonging to $W_g(m+\ell-j,n+j)$. It turns out that~{\eqref{eqn:g-words:ending:in:SgMg^k}} holds with
		$\LL_0=\LL_{2,0} \in W^S_g(m+\ell,n)$, 
		$\LL_{\ell-k+1}=\widehat{\LL}_{1,\ell-k+1}$,
		and
		$\LL_j=k\widehat{\LL}_{1,j}+\LL_{2,j}$, for $j=1,\dots,\ell-k$. And that ends the proof.\qedhere
	\end{enumerate}
\end{proof}

Before we deal with the remaining cases in Theorem \ref{thm:global:decomposition}
 we state two simple but very useful lemmas. The first one is easily proved by induction.

\begin{lemma}
Let $g\in\H(\D)$ and $n\in\N_0$. Then
\begin{equation}\label{eqn:commutator:STn}
T_g^nS_g=S_gT_g^n-nT_g^{n+1}\quad\mbox{on $\H_0(\D)$.}
\end{equation}
\end{lemma}

Identity~{\eqref{eqn:commutator:STn}} together with an induction argument will prove the following result.

\begin{lemma}\label{lem:difference:span}
Let $m,n\in\N$. If $L,L'\in W^T_g(m,n)$, then 
$L-L'$ is a linear combination of $g$-words in $W^T_g(m-1,n+1)$
with integer coefficients which do not depend on $g$.
\end{lemma}

\begin{proof}
Since $L-L'=(L-S^m_gT^n_g)-(L'-S^m_gT^n_g)$, we may 
assume that $L'=S^m_gT^n_g$. Now we proceed by induction on $m$. 

Assume that $L\in W^T_g(m,n)$. Then $L=T_g^{n_0}S_g\cdots T_g^{n_{m-1}}S_gT_g^{n_m+1}$, where $n_j\in\N_0$ and $n_0+\cdots+n_m+1=n$. 
When $m=1$, {\eqref{eqn:commutator:STn}}~implies that
\[
L-L'=(T_g^{n_0}S_g)T^{n_1}_gT_g-S_gT_g^n
    =-n_0\,T_g^{n+1}.
\]

Let $m\ge2$. Then~{\eqref{eqn:commutator:STn}} shows that
\[
L= (T_g^{n_0}S_g)T_g^{n_1}S_gT_g^{n_2}\cdots  
    S_gT_g^{n_m}T_g
 =S_gL_0-n_0L_1,
\]
with
$L_0=T_g^{n_0+n_1}S_gT_g^{n_2}\cdots S_gT_g^{n_m}T_g$
and
$L_1=T_g^{n_0+n_1+1}S_gT_g^{n_2}\cdots S_gT_g^{n_m}T_g$.
 
Since $L_0\in W^T_g(m-1,n)$, we may apply the induction 
hypothesis to obtain that $L_0-S_g^{m-1}T_g^n$ is a 
linear combination of $g$-words in $W^T_g(m-2,n+1)$, 
with integer coefficients which are independent of $g$. On the other hand,   
$L_1\in W^T_g(m-1,n+1)$, and it turns out that 
\[
L-L'=S_g(L_0-S_g^{m-1}T_g^n)-n_0L_1
\]
is a linear combination of $g$-words in 
$W^T_g(m-1,n+1)$ with integer coefficients which do not depend on $g$. And that ends the proof.
\end{proof}

An easy consequence of Lemma~{\ref{lem:difference:span}}
is the following general decomposition formula which will be an essential tool to complete the proof of Theorem~{\ref{thm:global:decomposition}.} 

\begin{proposition}\label{prop:operator:decomposition0}
Let $m\in\N_0$, $n\in\N$, and let $L_j\in W^T_g(m-j,n+j)$, for $j=0,\dots,m$. Then any  
$L\in W^T_g(m,n)$ can be decomposed as
\begin{equation}\label{eqn:prop:operator:decomposition0}
L=L_0+\sum_{j=1}^m c_j\,L_j,
\end{equation}
where the coefficients $c_j$ are integers which do not depend on $g$. 
\end{proposition}  
\begin{proof}
We proceed by induction on $m$. 
For $m=0$ it is clear that $L=L_0$. 
 Now assume that $m\ge1$. Then,
 by Lemma \ref{lem:difference:span},
 $L-L_0$ is a linear combination of $g$-words in 
$W^T_g(m-1,n+1)$ with integer coefficients which do not depend on $g$. By applying the induction hypothesis to each of those $g$-words, we deduce 
that~{\eqref{eqn:prop:operator:decomposition0}} holds for some coefficients $c_j\in\Z$ that do not depend on $g$.
 Hence the proof is complete.
\end{proof}

By applying Proposition~{\ref{prop:operator:decomposition0}} with
$L_j=S^{m-j}_gT^{n+j}_g$, $j=0,\dots,m$,  we 
obtain the following decomposition formula for $g$-words ending in $T_g$.
\begin{proposition}
\label{prop:operator:decomposition0T}
Let $m\in\N_0$ and $n\in\N$. Then any  
$L\in W^T_g(m,n)$ can be decomposed as
\begin{equation*}\label{eqn:prop:operator:decomposition0T}
L=S^m_gT^n_g+\sum_{j=1}^m c_j\,S^{m-j}_g\,T^{n+j}_g ,
\end{equation*}
where the coefficients $c_j$ are integers which do not depend on $g$. 
\end{proposition}
Now we deduce a similar decomposition formula for the $g$-words ending in $S_g$.
\begin{proposition}
\label{prop:operator:decomposition0S}
Let $m,n\in\N$. Then any  
$L\in W^S_g(m,n)$ can be decomposed as
\begin{equation}\label{eqn:prop:operator:decomposition0S}
L=S^m_g\,T^n_g\,\Pi_0+\sum_{j=1}^m c_j\,S^{m-j}_g\,T^{n+j}_g\,\Pi_0 ,
\end{equation}
where  the coefficients $c_j$ are integers which do not depend on $g$. 
\end{proposition}
 Proposition~{\ref{prop:operator:decomposition0S}} is a consequence of Proposition~{\ref{prop:operator:decomposition0T}} and the following two technical lemmas whose proofs, which we omit, are done by induction taking into account \eqref{eqn:commutator:STn}.

\begin{lemma}
Let $m\in\N_0$. Then 
\begin{equation}
\label{eqn:lem1:prop:operator:decomposition0S}
T_g\,S^m_g=(S_g-T_g)^m\,T_g\quad\mbox{on $\H_0(\D)$.}
\end{equation}
\end{lemma}

\begin{lemma}
Let $m\in\N$. Then 
\begin{equation}
\label{eqn:lem2:prop:operator:decomposition0S}
(S_g-T_g)^m
=\sum_{j=0}^m(-1)^{m-j}\,\tfrac{m!}{j!}\,S^j_g\,T^{m-j}_g
\quad\mbox{on $\H_0(\D)$.}
\end{equation}
\end{lemma}
\begin{proof}[{\bf Proof of Proposition~{\ref{prop:operator:decomposition0S}}}]
Let $L\in W^S_g(m,n)$. Then $L=\widetilde{L}\,T_g\,S^k_g$, where
$1\le k\le m$ and $\widetilde{L}$ is a $g$-word which only contains $m-k$ and $n-1$ letters $S_g$ and $T_g$, respectively. Since $S_g=S_g\Pi_0$, we get that
$L=\widetilde{L}\,T_g\,S^k_g\,\Pi_0$.
Now, taking into account that
$\Pi_0(\H(\D))\subset\H_0(\D)$, \eqref{eqn:lem1:prop:operator:decomposition0S} and~{\eqref{eqn:lem2:prop:operator:decomposition0S}} show that
\begin{equation*}
T_g\,S^k_g\,\Pi_0
=\sum_{i=0}^k(-1)^{k-i}\,\tfrac{k!}{i!}\,S^i_g\,T^{k+1-i}_g\,\Pi_0, 
\end{equation*}
and so 
\begin{equation}\label{eqn2:prop:operator:decomposition0S}
L=L_k\,\Pi_0+\sum_{i=0}^{k-1}(-1)^{k-i}\,\tfrac{k!}{i!}\,L_i\,\Pi_0
\end{equation}
where $L_i=\widetilde{L}\,S^i_g\,T^{k+1-i}_g\in W^T_g(m-k+i,n+k-i)$, for $i=0,\dots,k$.
 Then we may apply Proposition~{\ref{prop:operator:decomposition0T}} to decompose $L_i$ as
\begin{equation}\label{eqn3:prop:operator:decomposition0S}
L_i=
S^{m-k+i}_g\,T^{n+k-i}_g+\sum_{\ell=1}^{m-k+i}c_{i,\ell}\,S^{m-k+i-\ell}_g\,T^{n+k-i+\ell}_g,
\end{equation} 
where the coefficients 
$c_{i,\ell}$
 are integers which do not depend on $g$. 
By inserting in~{\eqref{eqn2:prop:operator:decomposition0S}} the expression of $L_i$ given by~{\eqref{eqn3:prop:operator:decomposition0S}},
 we deduce the decomposition~{\eqref{eqn:prop:operator:decomposition0S}}, and that completes the proof. 
\end{proof}

\begin{proof}[{\bf Proof of Theorem \ref{thm:global:decomposition}}]
By Proposition~{\ref{prop:delete:M:in:a:g-word}}, any $g$-word $L$ in $W_g(\ell,m,n)$, with $m+n\ge1$, decomposes as $L=L_0+L_1$,
where  $L_1$ is a linear combination of
$g$-words in  
${\displaystyle
\cup_{j=1}^{\ell}W_g(m+\ell-j,n+j)}$, with integer coefficients not depending on $g$.
Since, Proposition~{\ref{prop:delete:M:in:a:g-word}} also asserts that $L_0\in W^T_g(m+\ell,n)$, if 
$L$ ends in $T_gM_g^{i}$, for some $i\in \N_0$, and $L_0\in W^S_g(m+\ell,n)$, if 
$L$ ends in $S_gM_g^{i}$, for some $i\in \N_0$, 
 Propositions \ref{prop:operator:decomposition0T} and \ref{prop:operator:decomposition0S} complete the proof of Theorem~{\ref{thm:global:decomposition}}.	
\end{proof}

\subsection{A key lower norm estimate for some $g$-operators}
\label{subsection:lower:norm:estimate:g-operators}
The next result is crucial for our purposes. It is a quantitative optimized version of \cite[Theorem 1.1.a)]{Aleman:Cascante:Fabrega:Pascuas:Pelaez}. As what happens there, its proof depends on the notion of iterated commutators. We recall that the commutator of two linear operators  $A,B:\H(\D)\to\H(\D)$ is the linear operator $[A,B]:=AB-BA$. Then the iterated conmutators $[A,B]_k$, $k\in\N$, are defined inductively as follows:
\[
[A,B]_1:=[A,B]
\qquad\mbox{and}\qquad
[A,B]_{k+1}:=[[A,B]_k,B],\quad\mbox{for $k\in\N$.}
\]

On the other hand, from now on we will use repeatedly \cite[Proposition 4.3]{Aleman:Cascante:Fabrega:Pascuas:Pelaez}, which ensures that a  $g$-operator $L_g$ 
satisfies the identities
\begin{equation*}
\sup_{0<r<1}\|L_{g_r}\|_{\alpha,p}=\|L_g\|_{\alpha,p}=\liminf_{r\to1^{-}}\|L_{g_r}\|_{\alpha,p},
\end{equation*}
where $g_r(z)=g(rz)$.
An analysis of the proof of \cite[Proposition 4.3]{Aleman:Cascante:Fabrega:Pascuas:Pelaez} gives that 
this result also holds by replacing $A^p_{\alpha}$ and $\|\cdot\|_{\alpha,p}$ by $A^p_{\alpha}(0)$ and $\opnorm{\,\cdot\,}_{\alpha,p}$, respectively. Those identities together with Proposition \ref{prop:radialized:symbol}, \eqref{eqn:equiv:norms:Bq} and \eqref{eqn:BMOAqnorm:equal:boldnorm:-1} show that to prove Theorem \ref{thm:sharp:estimate:norm:word} we may assume that $g$ is analytic in a neighborhood
of the closed unit disc.
\begin{proposition}\label{prop:boundedness:Tg}
Let $L$ be a $g$-operator which satisfies
\begin{equation}\label{eqn:boundedness:Tg:0}
L=S_g^mT_g^n+\sum_{j=0}^{m-1}S_g^jT_gP_j(T_g)\quad\mbox{on $\H_0(\D)$},
\end{equation}
where $m,n\in\N$ and  $P_0,\dots,P_{m-1}$ are polynomials.
If $L\in\BB(A^p_{\alpha}(0))$, then $T_g\in\BB(A^p_{\alpha})$, and 
\begin{equation}\label{eqn:boundedness:Tg}
\|T_g\|_{\alpha,p}\le C\,\opnorm{L}_{\alpha,p}^{\frac1{m+n}},
\end{equation}
where $C>0$ is a constant which only depends on $\alpha,m,n$, and $p$.
\end{proposition}

In order to prove Proposition~{\ref{prop:boundedness:Tg}} we need the following three lemmas. The proof of the first one is easily done by induction.

\begin{lemma}
Let $g\in\H(\D)$ and let $g_0(z)=g(z)-g(0)$. Then
\begin{equation}\label{eqn:Tn:g0k}
T_g^n g_0^k=\frac{k!}{(n+k)!}\,g_0^{n+k}
\qquad(n,k\in \N_0)
\end{equation}
\end{lemma}

\begin{lemma} 
If $m,n\in\N$ and $m\ge n$, then
\begin{equation}\label{eqn:Jensen:inequality}
\|g^n\|_{\alpha,p}^{\frac1n}\le \|g^m\|_{\alpha,p}^{\frac1m}
\qquad(g\in\H(\overline{\D})).
\end{equation}
\end{lemma}

\begin{proof}
Since $\|g^\ell\|_{\alpha,p}^{\frac1\ell}=\|g\|_{\alpha,\ell p}$, for all $\ell\in\N$, the result follows from the fact that $\|g\|_{\alpha,q}$ is a non-decreasing function of $q\in(0,\infty)$.
\end{proof}

\begin{lemma}
Let $n\in\N$. Then
\begin{equation}\label{eqn:restricted:norm:Tn:0}
\opnorm{T^n_g}_{\alpha,p}\le
\|T^n_g\|_{\alpha,p}
\le c_p(2c_p+n+1)\,\opnorm{T^n_g}_{\alpha,p}
\qquad(g\in\H(\D)),    
\end{equation}
where $c_p=2^{\max(1,1/p)-1}$. 
 In particular, 
\begin{equation}\label{eqn:restricted:norm:Tn:00}
\opnorm{T^n_g}_{\alpha,p}\simeq\|T^n_g\|_{\alpha,p}
\simeq\|g\|^n_{\B^1_{\alpha}}\qquad(g\in\H(\D)).
\end{equation}
\end{lemma}

\begin{proof}
The left hand side inequality in \eqref{eqn:restricted:norm:Tn:0} is clear.
In order to prove the right hand side inequality we may assume, as usual, that $g\in\H(\overline{\D})$.
Let $f\in A^p_{\alpha}$. Then $f=f_0+f(0)$, where $f_0=f-f(0)\in A^p_{\alpha}(0)$. Since
\[
\|f_0\|_{\alpha,p}
\le c_p\,(\|f\|_{\alpha,p}+|f(0)|)
\le 2c_p\,\|f\|_{\alpha,p},
\]
it follows that 
\begin{align*}
\|T^n_gf\|_{\alpha,p}
&\le c_p\,
(\|T^n_gf_0\|_{\alpha,p}+|f(0)|\,\|T^n_g1\|_{\alpha,p})\\
&\le  c_p\,
(\opnorm{T^n_g}_{\alpha,p}\,\|f_0\|_{\alpha,p}
       +\|T^n_g1\|_{\alpha,p}\,\|f\|_{\alpha,p})\\
&\le  c_p\,
(2c_p\opnorm{T^n_g}_{\alpha,p}
       +\|T^n_g1\|_{\alpha,p})\,\|f\|_{\alpha,p},     
\end{align*}
and therefore
\begin{equation}\label{eqn:restricted:norm:Tn:1}
\|T^n_g\|_{\alpha,p}
\le c_p\,
(2c_p\opnorm{T^n_g}_{\alpha,p}
       +\|T^n_g1\|_{\alpha,p}).
\end{equation}
Now we want to estimate $\|T^n_g1\|_{\alpha,p}$ by $\opnorm{T^n_g}_{\alpha,p}$.
By \eqref{eqn:Tn:g0k}, 
$T^n_g1=\frac{g_0^n}{n!}$ and $T^n_gg_0=\frac{g_0^{n+1}}{(n+1)!}$, so \eqref{eqn:Jensen:inequality} shows that
\[
\|T^n_g1\|_{\alpha,p}=\frac{\|g_0^n\|_{\alpha,p}}{n!}
\le \frac{\|g_0^{n+1}\|_{\alpha,p}^{n/(n+1)}}{n!}
=\frac{((n+1)!)^{n/(n+1)}}{n!}\,\|T^n_gg_0\|_{\alpha,p}^{n/(n+1)}.
\]
Since $g\in\H(\overline{\D})$, it is clear that 
$g_0\in A^p_{\alpha}(0)$ and we get that 
\begin{equation}\label{eqn:restricted:norm:Tn:2}
\|T^n_g1\|_{\alpha,p}\le\frac{((n+1)!)^{n/(n+1)}}{n!}\,
\opnorm{T^n_g}_{\alpha,p}^{n/(n+1)}\|g_0\|_{\alpha,p}^{n/(n+1)}.  
\end{equation}
By applying \eqref{eqn:Jensen:inequality} again we have 
\[
\frac{\|g_0\|_{\alpha,p}^{n+1}}{(n+1)!}
\le \frac{\|g_0^{n+1}\|_{\alpha,p}}{(n+1)!}
=\|T^n_gg_0\|_{\alpha,p}\le\opnorm{T^n_g}_{\alpha,p}\|g_0\|_{\alpha,p}, 
\]
which implies that
\begin{equation}\label{eqn:restricted:norm:Tn:3}
\|g_0\|_{\alpha,p}^n\le (n+1)!\,\opnorm{T^n_g}_{\alpha,p}.
\end{equation}
Hence \eqref{eqn:restricted:norm:Tn:1}, \eqref{eqn:restricted:norm:Tn:2}, and \eqref{eqn:restricted:norm:Tn:3} show that the right hand side inequality in \eqref{eqn:restricted:norm:Tn:0} holds. Finally, \eqref{eqn:restricted:norm:Tn:00} is a direct consequence of \eqref{eqn:restricted:norm:Tn:0} and the estimates $\|T^n_g\|_{\alpha,p}
\simeq\|T_g\|^n_{\alpha,p}$ (see \cite[Proposition 4.1]{Aleman:Cascante:Fabrega:Pascuas:Pelaez}) and $\|T_g\|_{\alpha,p}
\simeq\|g\|_{\B^1_{\alpha}}$.
\end{proof}

\begin{proof}[{\bf Proof of Proposition~{\ref{prop:boundedness:Tg}}}]
By \cite[Proposition 4.3]{Aleman:Cascante:Fabrega:Pascuas:Pelaez}, it is clear that we only have to prove the estimate~{\eqref{eqn:boundedness:Tg}} for  $g\in\H(\overline{\D})$. Thus assume that $g\in\H(\overline{\D})$.
Since $P_j(T_g)$ commute  with $T_g$,
the formula
\[
[CAD,B]=C[A,B]D
\]
 valid for operators $A,B,C,D$ such that $C,D$ commute with $B$,
 \cite[Corollary~4.9]{Aleman:Cascante:Fabrega:Pascuas:Pelaez} 
 and the hypothesis~{\eqref{eqn:boundedness:Tg:0}} give that
\begin{equation*}
[L,T_g]_m= m!\,T_g^N
\quad\mbox{on $\H_0(\D)$,}
\end{equation*}
where $N=2m+n$.  
Now \eqref{eqn:restricted:norm:Tn:0}, (36) and Proposition~{4.1} in \cite{Aleman:Cascante:Fabrega:Pascuas:Pelaez} imply that 
\begin{align*}
\|T_g\|_{\alpha,p}^N
&\le c_1\,\|T_g^N\|_{\alpha,p}
\le c_2\, \opnorm{T_g^N}_{\alpha,p}
= m!\,c_2\,\opnorm{[L,T_g]_m}_{\alpha,p}\\
&\le 2^m m!\,c_2\,\opnorm{L}_{\alpha,p}\opnorm{T_g}_{\alpha,p}^m
\le 2^m m!\,c_2\,\opnorm{L}_{\alpha,p}\|T_g\|_{\alpha,p}^m,
\end{align*}
where $c_1,c_2>0$ are constants that only depend on $\alpha,N$ and $p$.
We conclude that $\|T_g\|_{\alpha,p}^{m+n}\le 2^mm!\,c_2\,\opnorm{L}_{\alpha,p}$, which ends the proof. 
\end{proof}

\subsection{Calder\' {o}n's area theorem and tent spaces}
\label{subsection:Calderon:thm}
\begin{definition}  
Let $\Gamma(\zeta)$ be the {\em Stolz region with vertex at $\zeta\in\T$} given by 
\begin{equation*}
\Gamma(\zeta):=\{z\in\D:|z-\zeta|<2(1-|z|)\},
\end{equation*}
and define
$\Gamma(\zeta)
:=|\zeta|\Gamma(\frac{\zeta}{|\zeta|})
=\{z\in\D:|z-\zeta|<2(|\zeta|-|z|)\}$, 
 for $\zeta\in\D\setminus\{0\}$. 

Let $\psi:\D\to\R$ be a continuous function. Then $|\nabla \psi|$ is a non-negative Borel measurable function on $\D$ (see \cite[page 11]{Vaisala}), and so we may consider the {\em square area function of $\psi$}, ${\mathcal S}\psi:\overline{\D}\setminus\{0\}\to[0,\infty]$, defined by
\begin{equation}\label{eqn:square:area:function:definition}
({\mathcal S}\psi)(\zeta):=
\biggl(
\int_{\Gamma(\zeta)}\bigl|\nabla \psi\bigr|^2\,dA
		\biggr)^{1/2}
		\qquad(\zeta\in\overline{\D}\setminus\{0\}).
\end{equation}
Recall that, 
 by the area formula \cite[Theorem 3.8]{Evans:Gariepy}, 
if $h\in\H(\D)$ then $({\mathcal S}|h|)(\zeta)^2$ is equal to $\frac1{\pi}$ times the  area (counting multiplicities) of the region $h(\Gamma(\zeta))$. This fact justifies the terminology ``square area function''.      
\end{definition}

\begin{namedthm*}{Calder\'{o}n's area theorem}[{\cite[Thm.\ 3]{Calderon}, \cite[Thm.\ 7.4]{Pavlovic}}]
Let $0<p,q<\infty$. Then there is a constant $C_{p,q}>0$ such that
\begin{equation}\label{eqn:Calderon:area:thm:0}
C_{p,q}^{-1}\|h\|_{H^p}^p
\le|h(0)|^p+\|{\mathcal S} |h|^q\|^{p/q}_{L^{p/q}(\T)}
\le C_{p,q}\|h\|_{H^p}^p
\qquad(h\in\H(\D)).
\end{equation}
\end{namedthm*}
Now we are going to give a version of this result for $A^p_{\alpha}$.
Let $0<p,q<\infty$ and let $h\in\H(\D)$. Then~\eqref{eqn:Calderon:area:thm:0} shows that 
\begin{equation}\label{eqn:Calderon:area:thm:1}
C_{p,q}^{-1}\|h_r\|_{H^p}^p
\le|h(0)|^p+\|{\mathcal S} |h_r|^q\|^{p/q}_{L^{p/q}(\T)}
\le C_{p,q}\|h_r\|_{H^p}^p
\quad(0<r<1).
\end{equation}
But, since $|\nabla|h_r|^q|(z)=r|\nabla|h|^q|(rz)$, for any $z\in\D$, we have that
\begin{equation*}
({\mathcal S} |h_r|^q)^2(\zeta)
=\int_{\Gamma(\zeta)}r^2|\nabla|h|^q|^2(rz)\,dA(z)
=\int_{\Gamma(r\zeta)}|\nabla|h|^q|^2\,dA
=({\mathcal S} |h|^q)^2(r\zeta),
\end{equation*}
for any $\zeta\in\T$.
Therefore, given $\alpha>-1$, we may integrate~{\eqref{eqn:Calderon:area:thm:1}} against $(\alpha+1)(1-r^2)^{\alpha}\,2r\,dr$ to obtain that
\begin{equation}\label{eqn:Calderon:area:thm:2}
C_{p,q}^{-1}\|h\|_{\alpha,p}^p
\le |h(0)|^p+\|{\mathcal S} |h|^q\|^{p/q}_{\alpha,p/q}
\le C_{p,q}\|h\|_{\alpha,p}^p
\qquad(h\in\H(\D)).
\end{equation}

Note that~{\eqref{eqn:Calderon:area:thm:0}} is just~{\eqref{eqn:Calderon:area:thm:2}} for $\alpha=-1$, and so~{\eqref{eqn:Calderon:area:thm:2}} holds for any $\alpha\ge-1$ and $0<p,q<\infty$.
Our next goal is writting~{\eqref{eqn:Calderon:area:thm:2}} in terms of tent spaces norms.
\begin{definition}
	Let $\alpha\ge-1$ and let $0<p,q<\infty$. 
For any  positive Borel measure $\nu$ on $\D$,  $T^p_{\alpha,q}(\nu)$ is the tent space of all Borel measurable functions $\psi$ on $\D$ such that 
\begin{equation}\label{eqn:tent:space:norm:definition}
	\|\psi\|^p_{T^p_{\alpha,q}(\nu)}:=\int_{\overline{\D}}\biggl(
	\int_{\Gamma(\zeta)}|\psi|^q\,d\nu\biggr)^{\frac{p}{q}}\,dA_{\alpha}(\zeta)<\infty,
\end{equation}
where both $dA_{\alpha}(\zeta)=(\alpha+1)(1-|\zeta|^2)^{\alpha}\,dA(\zeta)$, for $\alpha>-1$, and $dA_{-1}:=d\sigma$ (the normalized arc-length measure on the unit circle)  are considered as positive measures on $\overline{\D}$.
	When $d\nu=dA$, the corresponding tent space is simply denoted by $T^p_{\alpha,q}$.  	It is clear  that
	\begin{equation}\label{eqn:Tpq:norm:of:a:power}
			\||\psi|^r\|_{T^p_{\alpha,q}(\nu)}=\|\psi\|^r_{T^{pr}_{\alpha,qr}(\nu)}
			\qquad(0<r<\infty).
		\end{equation}
	Moreover, for every $\alpha\ge-1$ and $0<p<\infty$, there is a constant $C_{\alpha,p}>0$, only depending on $\alpha$ and $p$, such that
	\begin{equation}\label{eqn:Tpp(nu)=Lp(rhodnu)}
			C_{\alpha,p}^{-1}\,\|\psi\|_{L^p(\rho^{\alpha+2}\,d\nu)}\le\|\psi\|_{T^p_{\alpha,p}(\nu)}\le C_{\alpha,p}\,\|\psi\|_{L^p(\rho^{\alpha+2}\,d\nu)}.
		\end{equation} 
\end{definition}

H\"{o}lder's inequality easily implies the following useful result:
\begin{proposition}\label{prop:Holder:ineq:tent:spaces}
	Let $0<p_0,p_1,\dots,p_n,q_0,q_1,\dots,q_n<\infty$ such that
	\[
	\frac1{p_0}=\frac1{p_1}+\cdots+\frac1{p_n}
	\qquad\mbox{and}\qquad
	\frac1{q_0}=\frac1{q_1}+\cdots+\frac1{q_n}.
	\]
	If $\nu$ is a positive Borel measure on $\D$ and
	$\psi_j\in T^{p_j}_{\alpha,q_j}(\nu)$, for $j=1,\dots,n$, then $\psi_0=\psi_1\cdots\psi_n\in T^{p_0}_{\alpha,q_0}(\nu)$ and
	$\|\psi_0\|_{T^{p_0}_{\alpha,q_0}(\nu)}
	\le \|\psi_1\|_{T^{p_1}_{\alpha,q_1}(\nu)}\cdots\|\psi_n\|_{T^{p_n}_{\alpha,q_n}(\nu)}$.
\end{proposition}

Note that~{\eqref{eqn:square:area:function:definition}} and~{\eqref{eqn:tent:space:norm:definition}} show that 
$\|{\mathcal S} |h|^q\|^{p/q}_{\alpha,p/q}=
\||\nabla|h|^q|\|^{p/q}_{T^{p/q}_{\alpha,2}}$, so~{\eqref{eqn:Calderon:area:thm:2}} can be written as the following result.
\begin{theorem}
Let $0<p,q<\infty$. Then there is a constant $C_{p,q}>0$ such that, for any $\alpha\ge-1$, satisfies that
\begin{equation}\label{eqn:Calderon:area:thm:3}
C_{p,q}^{-1}\|h\|_{\alpha,p}^p
\le |h(0)|^p+\||\nabla|h|^q|\|^{p/q}_{T^{p/q}_{\alpha,2}}
\le C_{p,q}\|h\|_{\alpha,p}^p
\qquad(h\in\H(\D)).
\end{equation}
\end{theorem}
\begin{corollary}
For any $0<p<\infty$ there is a constant $C_p>0$ such that
\begin{equation}\label{eqn:estimate:Apalpha:norm:derivative}
C_p^{-1}\|h\|_{\alpha,p}^p\le
|h(0)|^p+\|h'\|^p_{T^p_{\alpha,2}}
\le	C_p\|h\|_{\alpha,p}^p
\qquad(h\in\H(\D),\alpha\ge-1).
\end{equation}
\end{corollary}

Finally, we recall that the $(\alpha+2)$-Carleson measure norm of a positive Borel measure on $\D$ can be estimated by using tent spaces as follows:
\begin{theorem}
	\label{thm:Carleson:measure:norm:tent:spaces}
	For $\alpha\ge-1$, $0<p,q<\infty$ and any positive Borel measure $\nu$ on $\D$, define
	\begin{equation}\label{eqn:definition:Mpqnu}
			M_{p,q}(\nu,\alpha):=\sup\{\|f\|^q_{T^p_{\alpha,q}(\nu)}:
			f\in A^p_{\alpha},\,\|f\|_{\alpha,p}=1\,\}.
		\end{equation}
	Let $d\mu^{\alpha}=\rho^{\alpha+2}\,d\nu$, where $\rho(z):=1-|z|^2$. If $\nu(\{0\})=0$, then there exists a constant $C_{\alpha,p,q}>0$, only depending on $\alpha$, $p$ and $q$, such that 
	\begin{equation}\label{eqn:Carleson:measure:norm:tent:spaces}
			C^{-1}_{\alpha,p,q}\,\|\mu^{\alpha}\|_{\mathcal{C(\alpha)}}
			\le M_{p,q}(\nu,\alpha)\le
			C_{\alpha,p,q}\,\|\mu^{\alpha}\|_{\mathcal{C(\alpha)}}.
		\end{equation}
\end{theorem}

Theorem~{\ref{thm:Carleson:measure:norm:tent:spaces}} is proved 
in~{\cite[Theorem 1]{Cohn}} (see also~{\cite[Theorem 3.1]{Gong:Lou:Wu}) for $\alpha=-1$ and $q=1$, but the same proof gives the result for any $q>0$ (see also~{\cite[Theorem 9]{Pelaez:Rattya:Sierra}} for an alternative proof). For $\alpha>-1$, Theorem~{\ref{thm:Carleson:measure:norm:tent:spaces}} is a particular case of~{\cite[Theorem 1]{Pelaez:Rattya:Sierra}}.

Taking into account \eqref{eqn:Tpp(nu)=Lp(rhodnu)}, the estimate~{\eqref{eqn:Carleson:measure:norm:tent:spaces}} for $p=q$	is just the well-known estimate
		\[
		C^{-1}_{\alpha,p}\,\|\mu\|_{\mathcal{C(\alpha)}}
		\le \sup\{\|f\|^p_{L^p(\mu)}:
		f\in A^p_{\alpha},\,\|f\|_{\alpha,p}=1\,\} \le
		C_{\alpha,p}\,\|\mu\|_{\mathcal{C(\alpha)}},
		\]
		where $C_{\alpha,p}>0$ is a constant only depending on $\alpha$ and $p$.	
				
Note that if 
\begin{equation}\label{eqn:def:nugs}
d\nu_{g,s}:=|\nabla|g|^s|^2\,dA
\qquad(g\in\H(\D),\,s>0),
\end{equation}
then $\rho^{\alpha+2}\,d\nu_{g,s}$  is the measure $d\mu^{\alpha}_{g,s}$ defined by~{\eqref{eqn:Carleson:measure:Balphaq}}. 
Therefore, by applying Theorem~{\ref{thm:Carleson:measure:norm:tent:spaces}} and \eqref{eqn:derivative:symbol:Carleson:measure}, we  
obtain the useful estimate 
\begin{equation}
\label{eqn:Mpq:simeq:BMOAs-norm}
M_{p,q}(\nu_{g,s},\alpha)\simeq\boldnorm{g}_{\alpha,s}^{2s}
				\qquad(g\in\H(\D),\,s>0),
			\end{equation} 
		which holds for every $\alpha\ge-1$ and $0<p,q<\infty$. 

\subsection{Boundedness of the maximal operator from  $A^p_{\alpha}$ to $L^p_{\alpha}$}

\begin{definition}
The non-tangential maximal function of a measurable function $\psi$ on $\D$ is defined by
\begin{equation}\label{eqn:definition:non-tangential:maximal:function}
	({\mathcal M}
	\psi)(\zeta):=\sup_{z\in\Gamma(\zeta)}|\psi(z)|
	\qquad(\zeta\in\overline{\D}).
\end{equation}
\end{definition}
A well-known and important property of the non-tangential maximal operator ${\mathcal M}$ is the following: 
\begin{theorem}{\cite[Theorem II.3.1]{Garnett}}
${\mathcal M}$ is bounded from $H^p$ to $L^p(\T)$, for any $0<p<\infty$. Indeed,  there is a constant $C>0$ such that
\begin{equation}\label{eqn:boundedness:non-tangential:maximal:function}
\|{\mathcal M}f\|^p_{L^p(\T)}\le C\,\|f\|^p_{H^p}\qquad(f\in H^p,0<p<\infty). 
\end{equation}
\end{theorem}
\begin{corollary}\label{cor:boundedness:non-tangential:maximal:function}
${\mathcal M}$ is bounded from  $A^p_{\alpha}$ to $L^p_{\alpha}$, for every $\alpha\ge-1$ and $0<p<\infty$. Indeed, there is a constant $C>0$ such that
\begin{equation}
\label{eqn:boundedness:non-tangential:maximal:function:1}
\|{\mathcal M}f\|^p_{\alpha,p}\le C\,\|f\|^p_{\alpha,p}
\qquad(f\in A^p_{\alpha},\alpha\ge-1,0<p<\infty).
\end{equation}
\end{corollary}
 Corollary~{\ref{cor:boundedness:non-tangential:maximal:function}} is a particular case of the more general result \cite[Lemma 4.4]{Pelaez:Rattya}, but we note that \eqref{eqn:boundedness:non-tangential:maximal:function:1} is an immediate consequence of \eqref{eqn:boundedness:non-tangential:maximal:function} using the same argument that proves \eqref{eqn:Calderon:area:thm:2} from \eqref{eqn:Calderon:area:thm:0}.

\section{ Proof of Theorem~{\ref{thm:sharp:estimate:norm:word}}~{\sf a)}}
\label{section:easycase}

In this section we prove the first part of Theorem~{\ref{thm:sharp:estimate:norm:word}}. 

First of all, note that $W_g(\ell,0,0)=\{M^{\ell}_g\}$, for any $\ell\in\N$.
Since $M^{\ell}_g=M_{g^{\ell}}$, we have that $\|M^{\ell}_g\|_{\alpha,p}=\|g\|_{H^{\infty}}^{\ell}=\opnorm{M^{\ell}_g}_{\alpha,p}$, so from now on we may assume that $L\in W_g(\ell,m,0)$, with $\ell\in\N_0$ and $m\in\N$. 
In this case, Theorem \ref{thm:global:decomposition} shows that
\begin{equation}\label{eqn:decomposition:without:T}
L= S^N_g\Pi_0 +\sum_{j=1}^Na_j\,S_g^{N-j}T_g^j
+\sum_{j=1}^Nb_j\,S_g^{N-j}T_g^j\Pi_0,
\end{equation}
where  $N=\ell+m\in\N$, and the coefficients $a_j,b_j$ are integers which do not depend on $g$.

If $g\in H^{\infty}$ then $\|S_g^{N-j}\|_{\alpha,p}\lesssim\|g\|_{H^{\infty}}^{N-j}$ and 
$\|T_g^j\|\lesssim\|g\|_{\B^{1}_{\alpha}}^j\lesssim\|g\|_{H^{\infty}}^j$, for $j=0,\dots,N$, so 
$\opnorm{L}_{\alpha,p}\le \|L\|_{\alpha,p}\lesssim\|g\|_{H^{\infty}}^N$, since $\Pi_0$ is bounded on $A^p_{\alpha}$.

Now we want to prove the estimate~{$\|g\|_{H^{\infty}}^N\lesssim\opnorm{L}_{\alpha,p}$,}
or equivalently $\sup_{0<r<1}\|g_r\|^N_{H^{\infty}}\lesssim\opnorm{L}_{\alpha,p}$, where $g_r(z)=g(rz)$. Assume that $L\in\BB(A^p_{\alpha}(0))$. 
 Note that \eqref{eqn:decomposition:without:T} gives that
\begin{equation*}
	L_g= S^N_g +\sum_{j=1}^Nc_j\,S_g^{N-j}T_g^j\quad\mbox{on $\H_0(\D)$,}
\end{equation*}
where the $c_j$'s are integers not depending on $g$. By~{\cite[Proposition 4.3]{Aleman:Cascante:Fabrega:Pascuas:Pelaez}} and  Proposition~{\ref{prop:boundedness:Tg}}, it follows that
\begin{align*}
\|g_r\|^N_{H^{\infty}}
&\lesssim\|S^N_{g_r}\|_{\alpha,p}=
\opnorm{S^N_{g_r}}_{\alpha,p}
\lesssim
\opnorm{L_{g_r}}_{\alpha,p}
+\sum_{j=1}^N \opnorm{S^{N-j}_{g_r}T^j_{g_r}}_{\alpha,p}
\\
&\le
\opnorm{L_g}_{\alpha,p}
+\sum_{j=1}^N \|S^{N-j}_{g_r}T^j_{g_r}\|_{\alpha,p}
\\	
&\le
\opnorm{L_g}_{\alpha,p}
+\|T_{g_r}\|_{\alpha,p}\sum_{j=1}^N \|S_{g_r}\|^{N-j}_{\alpha,p}\|T_{g_r}\|_{\alpha,p}^{j-1}\\
&\lesssim
\opnorm{L_g}_{\alpha,p}
+\opnorm{L_{g}}_{\alpha,p}^{\frac1N}\,\|g_r\|^{N-1}_{H^{\infty}}.
\end{align*}
Therefore either $\opnorm{L_g}_{\alpha,p}=\|g_r\|_{H^{\infty}}=0$ 
or $0<\opnorm{L_g}_{\alpha,p}<\infty$ and 
\begin{equation*}
	\frac{\|g_r\|_{H^{\infty}}^N}{\opnorm{L_g}_{\alpha,p}}
	\lesssim 
	1+\biggl(\frac{\|g_r\|^N_{H^{\infty}}}{\opnorm{L_g}_{\alpha,p}}\biggr)^{\frac{N-1}N}.
\end{equation*}
Hence we conclude that
$\sup_{0<r<1}\|g_r\|^N_{H^{\infty}}\lesssim\opnorm{L_g}_{\alpha,p}$, and Theorem~{\ref{thm:sharp:estimate:norm:word}~{\sf a)}} is proved.

\section{ Proof of Theorem~{\ref{thm:sharp:estimate:norm:word}~{\sf b)}}: Reduction to the case 
	$L=S^m_gT^n_g$}

In this section we will deduce Theorem~{\ref{thm:sharp:estimate:norm:word}~{\sf b)}} from the case $L=S^m_gT^n_g$ $m\in\N_0$, $n\in\N$, that will be proved in the following four sections.

First of all the following remark is in order:

\begin{remark}\label{remark:SmTn:trivial:cases}
Theorem~{\ref{thm:sharp:estimate:norm:word}~{\sf b)}} holds for $L=S^m_gT^n_g$ when either $m=0$ or $n=1$,
 since we know that~{\eqref{eqn:restricted:norm:Tn:00}} holds and  $S^m_gT_g=\frac1{m+1}\,T_{g^{m+1}}$.
\end{remark}

Assume Theorem~{\ref{thm:sharp:estimate:norm:word}~{\sf b)}} holds for $L=S^m_gT^n_g$ $m\in\N_0$, $n\in\N$, that is,
\begin{equation}
	\label{eqn:sharp:estimate:SmTn}
	\|S^m_gT^n_g\|_{\alpha,p}\simeq\opnorm{S^m_gT^n_g}_{\alpha,p}
	\simeq\|g\|^{m+n}_{\B^s_{\alpha}}
	\qquad(g\in\H(\D)),
\end{equation}
where $s=\frac{m}n+1$, and we want to prove Theorem~{\ref{thm:sharp:estimate:norm:word}~{\sf b)}} for any $L\in W_g(\ell,m,n)$, with $\ell,m\in\N_0$ and $n\in\N$, that is,
\begin{equation}
	\label{eqn:sharp:estimate:norm:word:0}
	\|L\|_{\alpha,p}\simeq\opnorm{L}_{\alpha,p}
	\simeq\|g\|^{N}_{\B^s_{\alpha}}
		\qquad(g\in\H(\D)),
\end{equation}
where $N=\ell+m+n$ and $s=\frac{\ell+m}n+1$. By Remark \ref{remark:SmTn:trivial:cases}, we may assume that $L\in W_g(\ell,m,n)$, where $\ell,m,n\in\N_0$ so that $n\ge 1$ and $k=\ell+m\ge1$. Let $N=k+n$. 
 Then $L$  satisfies~{\eqref{thm:global:decomposition}, and so, taking into account that $\Pi_0\in\BB(A^p_{\alpha})$,
 \eqref{eqn:sharp:estimate:SmTn} and~{\eqref{eqn:key:estimate:Carleson:norms:symbols1}} give that
\begin{equation*}
\|L\|_{\alpha,p}
\lesssim
\|S^k_gT^n_g\|_{\alpha,p}
+\sum_{j=1}^k \|S^{k-j}_gT^{n+j}_g\|_{\alpha,p}
\lesssim
\sum_{j=0}^k \|g\|^N_{\B^{\frac{k+n}{n+j}}_{\alpha}}
\lesssim
\|g\|^N_{\B^{s}_{\alpha}}.
\end{equation*}¡
Therefore
\begin{equation}
\label{eqn:sharp:upper:estimate:norm:word}
\|L\|_{\alpha,p}
\lesssim\|g\|^N_{\B^{s}_{\alpha}}
\qquad(g\in\H(\D),\,L\in W_g(\ell,m,n)).
\end{equation}
In order to prove the reverse estimate,  we may assume that $g\in\H(\overline{\D})$, by~{\cite[Proposition 4.3]{Aleman:Cascante:Fabrega:Pascuas:Pelaez}} and  
Proposition~{\ref{prop:radialized:symbol}}.
Then \eqref{eqn:sharp:estimate:SmTn},
\eqref{eqn:global:decomposition:H0}, Proposition  \ref{prop:boundedness:Tg} together with  \eqref{eqn:key:estimate:Carleson:norms:symbols1} show that
\begin{align*}
\|g\|^N_{\B_{\alpha}^s}       
&\lesssim
\opnorm{S^k_gT^n_g}_{\alpha,p}
\lesssim
\opnorm{L}_{\alpha,p}
+\sum_{j=1}^k \|S^{k-j}_gT^{n+j}_g\|_{\alpha,p}\\
&\le
\opnorm{L}_{\alpha,p}
+\|T_g\|_{\alpha,p}\sum_{j=1}^k \|S^{k-j}_gT^{n+j-1}_g\|_{\alpha,p}\\
&\lesssim
\opnorm{L}_{\alpha,p}
+\opnorm{L}_{\alpha,p}^{\frac1N}\sum_{j=1}^m \|g\|^{N-1}_{\B^{\frac{N-1}{n+j-1}}_{\alpha}}\\
&\lesssim
\opnorm{L}_{\alpha,p}
+\opnorm{L}_{\alpha,p}^{\frac1N}\,\|g\|^{N-1}_{\B_{\alpha}^s},
\end{align*}
It turns out that either $\opnorm{L}_{\alpha,p}=\|g\|_{\B_{\alpha}^s}=0$ 
or $0<\opnorm{L}_{\alpha,p}<\infty$ and 
\begin{equation*}
\frac{\|g\|^N_{\B_{\alpha}^s}}{\opnorm{L}_{\alpha,p}}
\lesssim 
1+\biggl(\frac{\|g\|^N_{\B_{\alpha}^s}}{\opnorm{L}_{\alpha,p}}\biggr)^{\frac{N-1}N}.
\end{equation*}
Hence we have that
\begin{equation}
\label{eqn:sharp:lower:estimate:norm:word}
\|g\|^N_{\B_{\alpha}^s}
\lesssim
\opnorm{L}_{\alpha,p}
\qquad(g\in\H(\D),\,L\in W_g(\ell,m,n)).
\end{equation}
Finally, it is clear that \eqref{eqn:sharp:upper:estimate:norm:word} and \eqref{eqn:sharp:lower:estimate:norm:word} give~{\eqref{eqn:sharp:estimate:norm:word:0}}.
\vspace*{6pt}

The next four sections will be devoted to prove Theorem~{\ref{thm:sharp:estimate:norm:word}~{\sf b)}} for $L=S^m_gT^n_g$ $m\in\N_0$, $n\in\N$.

\section{Proof of Theorem~{\ref{thm:sharp:estimate:norm:word} {\sf b)}} for $L=S^m_gT^n_g$: Sufficient condition}

 Taking into account Remark \ref{remark:SmTn:trivial:cases}, the sufficiency in Theorem~{\ref{thm:sharp:estimate:norm:word} {\sf b)}} reduces to
\begin{theorem}\label{thm:sharp:upper:estimate:norm:word:ST:Hp}
	Let $m\in\N$, $n\in\N$, $n\ge2$, and $s=\frac{m}n+1$. Then
	\begin{equation*}
		\|S^m_gT^n_g\|_{\alpha,p}
		\lesssim\|g\|^{sn}_{\B^s_{\alpha}}
		\qquad(g\in\H(\D)).
	\end{equation*}
\end{theorem}

It is worth mentioning that Theorem~\ref{thm:sharp:upper:estimate:norm:word:ST:Hp} has been used to prove both inequalities
\eqref{eqn:sharp:lower:estimate:norm:word} and \eqref{eqn:sharp:upper:estimate:norm:word}.
In order to show this result we need to consider the following operator:

\begin{definition}
	For $g\in\H(\D)$, $\tau\in\Q$, $\tau>0$, and $\ell\in\N$, we define
	\begin{equation*}
		Q^{\tau,\ell}_g f:=|g|^{\tau\ell}\,T^{\ell}_gf
		\qquad(f\in\H(\D)).
	\end{equation*}
	Note that, for $g\in\H(\overline{\D})$,  $T_g^{\ell}\in\BB(A^p_{\alpha})$ and $|g|^{\tau\ell}\in L^{\infty}(\overline{\D})$, so it follows that $Q_g^{\tau,\ell}$ is a bounded linear operator from $A^p_{\alpha}$ to $L^p_{\alpha}$, for any $0<p<\infty$ and $\alpha\ge-1$. 
	In particular, for $g\in\H(\overline{\D})$, it makes sense to consider
	\begin{align*}
		\|Q_g^{\tau,\ell}\|_{\alpha,p}
		&:= 
		\sup\bigl\{\|Q_g^{\tau,\ell}f\|_{\alpha,p}:f\in A^p_{\alpha},\,\|f\|_{\alpha,p}=1\bigr\}\quad\mbox{and}
		\\
		\opnorm{Q_g^{\tau,\ell}}_{\alpha,p}
		&:= 
		\sup\bigl\{\|Q_g^{\tau,\ell}f\|_{\alpha,p}:f\in A^p_{\alpha}(0),\,\|f\|_{\alpha,p}=1\bigr\}.
	\end{align*}
For a sake of convenience, $Q^{\tau,0}_g$ will be the identity operator on $\H(\D)$.
\end{definition}
\vspace*{6pt}

The proof of Theorem~{\ref{thm:sharp:upper:estimate:norm:word:ST:Hp}} is a direct consequence of the following two fundamental results.

\begin{proposition}
	Let $m,n\in\N$, $n>1$, $\tau=\frac{m}n$ and $s=\tau+1$. Then
	\begin{equation}
		\label{eqn:estimate:norm:SmTn:by:norm:Q:Hp}
		\|S^m_gT^n_g\|_{\alpha,p}\lesssim\,\|g\|^s_{\B^s_{\alpha}} 
		\,\|Q^{\tau,n-1}_g\|_{\alpha,p}
		\qquad(g\in\H(\overline{\D})).
	\end{equation}
\end{proposition}

\begin{proof}
	Let $f\in A^p_{\alpha}$.
	Then~{\eqref{eqn:estimate:Apalpha:norm:derivative}}, \eqref{eqn:Tpq:norm:of:a:power}, and~{\eqref{eqn:Mpq:simeq:BMOAs-norm}} show that
	\begin{align*}
		\|S^m_gT^n_gf\|_{\alpha,p}
		&\lesssim \|g^mg'T^{n-1}_gf\|_{T^p_{\alpha,2}}
		=\||g|^{\frac{m}n}g'|g|^{m\frac{n-1}n}|T^{n-1}_gf|\|_{T^p_{\alpha,2}}
		\\
		&=\||h|^{\frac1n}\|_{T^p_{\alpha,2}(\nu_{g,s})}
		=\|h\|^{\frac1n}_{T^{\frac{p}n}_{\alpha,\frac2n}(\nu_{g,s})}
		\\
		&\lesssim
		\|g\|^s_{\B^s_{\alpha}}\,\|h\|^{\frac1n}_{\alpha,\frac{p}n}
		=\|g\|^s_{\B^s_{\alpha}}\,\|Q_g^{\tau,n-1}f\|_{\alpha,p},
	\end{align*}
	where $h=g^{m(n-1)}\,(T^{n-1}_gf)^n$.
	Therefore
	estimate~{\eqref{eqn:estimate:norm:SmTn:by:norm:Q:Hp}} holds, and that completes the proof. 
\end{proof}

\begin{proposition}
	Let $\ell\in\N$, $\tau\in\Q$, $\tau>0$, and $s=\tau+1$. Then
	\begin{equation}\label{eqn:upper:estimate:norm:Q:Hp}
		\|Q^{\tau,\ell}_g\|_{\alpha,p}\lesssim\|g\|^{s\ell}_{\B^{s}_{\alpha}}
		\qquad(g\in\H(\overline{\D})). 
	\end{equation}
\end{proposition}

\begin{proof}
	Let us write $\tau$ as a fraction $m/n$, where $m,n\in\N$ and $m,n>2$. Let $f\in A^p_{\alpha}$ such that $\|f\|_{\alpha,p}=1$.
	Since 
	$\|Q^{\tau,\ell}_gf\|_{\alpha,p}=\|g^{m\ell}(T^{\ell}_g f)^n\|^ {1/n}_{\alpha,p/n}$ and
	\[
	\bigl(g^{m\ell}(T^{\ell}_g f)^{n}\bigr)'
	=m\ell g^{m\ell-1}g'(T_g^{\ell}f)^{n}+ng^{m\ell}g'(T^{\ell-1}_gf)(T^{\ell}_gf)^{n-1},
	\] 
	by~{\eqref{eqn:estimate:Apalpha:norm:derivative}} we have that
	\begin{equation}\label{eqn:estimate:norm:Qf:Hp}
		\|Q^{\tau,\ell}_gf\|_{\alpha,p}\lesssim A_f+B_f,
	\end{equation}
	where
	$A_f=\|g^{m\ell-1}g'(T_g^{\ell}f)^{n}\|^{1/n}_{T^{p/n}_{\alpha,2}}$
	and
	$B_f=\|g^{m\ell}g'(T^{\ell-1}_gf)(T^{\ell}_gf)^{n-1}\|^{1/n}_{T^{p/n}_{\alpha,2}}$.\vspace*{3pt}\newline
	In order to estimate $A_f$,
	observe that 
	\begin{equation*}
		A_f
		=\||g|^{\tau}g'|g|^{m\ell-s}|T_g^{\ell}f|^{n}\|^{1/n}_{T^{p/n}_{\alpha,2}}
		\simeq\||g|^{m\ell-s}|T_g^{\ell}f|^{n}\|^{1/n}_{T^{p/n}_{\alpha,2}(\nu_{g,s})}.
	\end{equation*}
	Now $|g|^{m\ell-s}|T_g^{\ell}f|^{n}=|h|^{1/n}$, where $h=g^{mn\ell-m-n}(T^{\ell}_gf)^{n^2}\in\H(\D)$, since $mn\ell-m-n\in\N$ (because $m,n>2$). 
	Then~{\eqref{eqn:Tpq:norm:of:a:power}} and~{\eqref{eqn:Mpq:simeq:BMOAs-norm}} give that
	\begin{equation*}
		A_f=\|h\|^{1/n^2}_{T^{p/n^2}_{\alpha,2/n}(\nu_{g,s})}
		\lesssim\|g\|_{\B^s_{\alpha}}^{s/n}\,\|h\|_{\alpha,p/n^2}^{1/n^2}.
	\end{equation*}
	Since $m,n>2$, $m\ell\ge m>s$, so we may  apply H\"older's inequality with exponents $\frac{m\ell}{m\ell-s}$ and $\frac{m\ell}{s}$ and take into account the estimates $\|T_g^{\ell}\|_{\alpha,p}\lesssim\|g\|_{\B^1_{\alpha}}^{\ell}$ and~{\eqref{eqn:BMOAq:Bq:nested:norms}} to obtain that 
	\begin{align*}
		\|h\|_{\alpha,p/n^2}^{1/n^2}
		&=\||g|^{m\ell-s}|T^{\ell}_gf|^{n}\|_{\alpha,p/n}^{1/n}
		\le \||g|^{m\ell}|T^{\ell}_gf|^{n}\|_{\alpha,p/n}^{\frac{m\ell-s}{mn\ell}}\,
		\||T_g^{\ell}f|^{n}\|_{\alpha,p/n}^{\frac{s}{mn\ell}}
		\\
		&=\|Q_g^{\tau,\ell}f\|_{\alpha,p}^{\frac{m\ell-s}{m\ell}}\,
		\|T_g^{\ell}f\|_{\alpha,p}^{\frac{s}{m\ell}}
		\lesssim \|Q_g^{\tau,\ell}\|_{\alpha,p}^{\frac{m\ell-s}{m\ell}}\,
		\|g\|_{\B^1_{\alpha}}^{s/m}
		\\
		&\lesssim
		\|Q_g^{\tau,\ell}\|_{\alpha,p}^{\frac{m\ell-s}{m\ell}}\,
		\|g\|_{\B^s_{\alpha}}^{s/m}.
	\end{align*}
	Since $\frac{s}{n}+\frac{s}{m}=\frac{s}{m}(\tau+1)=\frac{s^2}{m}$, it follows that 
	\begin{equation}\label{eqn:estimate:A:Hp}
		A_f\lesssim
		\|g\|_{\B^s_{\alpha}}^{s^2/m}\,\|Q_g^{\tau,\ell}\|_{\alpha,p}^{1-s/(m\ell)}.
	\end{equation}
	
	\vspace*{6pt}
	
	We estimate $B_f$ similarly. Since $mn\ell-m\in\N$, 
	\[
	h=g^{mn\ell-m}(T^{\ell-1}_gf)^n(T^{\ell}_gf)^{n^2-n}\in \H(\D),
	\]
	and so~{\eqref{eqn:Tpq:norm:of:a:power}} and~{\eqref{eqn:Mpq:simeq:BMOAs-norm}} imply that
	\begin{align*}
		B_f
		&=\||g|^{\tau}g'|g|^{m\ell-\tau}|T^{\ell-1}_gf||T^{\ell}_gf|^{n-1}\|^{1/n}_{T^{p/n}_{\alpha,2}}
		\\
		&=\||g|^{m\ell-\tau}|T^{\ell-1}_gf||T^{\ell}_gf|^{n-1}\|^{1/n}_{T^{p/n}_{\alpha,2}(\nu_{g,s})}
		=\||h|^{1/n}\|^{1/n}_{T^{p/n}_{\alpha,2}(\nu_{g,s})}
		=\|h\|^{1/n^2}_{T^{p/n^2}_{\alpha,2/n}(\nu_{g,s})}
		\\
		&\lesssim\|g\|_{\B^s_{\alpha}}^{s/n}\,\|h\|_{\alpha,p/n^2}^{1/n^2}
		=\|g\|_{\B^s_{\alpha}}^{s/n}\,\||g|^{m\ell-\tau}|T^{\ell-1}_gf||T^{\ell}_gf|^{n-1}\|_{\alpha,p/n}^{1/n}.
	\end{align*}
	Now
	$m\ell-\tau=m\ell-\tau\ell+\tau\ell-\tau=\tau\ell(n-1)+\tau(\ell-1)$,
	so we may apply H\"{o}lder's inequality with exponents $n$ and $\frac{n}{n-1}$ to get that
	\begin{align*}
		B_f
		&\lesssim\|g\|^{s/n}_{\B^s_{\alpha}}\,
		\||g|^{\tau(\ell-1)}|T^{\ell-1}_gf|\,
		|g|^{\tau\ell(n-1)}|T^{\ell}_gf|^{n-1}\|_{\alpha,p/n}^{1/n}\\
		&\le\|g\|^{s/n}_{\B^s_{\alpha}}\,
		\||g|^{n\tau(\ell-1)}|T^{\ell-1}_gf|^{n}\|^{1/n^2}_{\alpha,p/n}\, 
		\||g|^{n\tau\ell}|T^{\ell}_gf|^{n}\|^{(n-1)/n^2}_{\alpha,p/n} \\
		&=\|g\|^{s/n}_{\B^s_{\alpha}}\,
		\|Q_g^{\tau,\ell-1}f\|^{1/n}_{\alpha,p}\,
		\|Q_g^{\tau,\ell}f\|^{1-1/n}_{\alpha,p}.
	\end{align*}
	It follows that
	\begin{equation}\label{eqn:estimate:B:H^p}
		B_f\lesssim\left\{
		\begin{array}{ll}
			\|g\|^{s/n}_{\B^s_{\alpha}}\,\,
			\|Q^{\tau,1}_g\|^{1-1/n}_{\alpha,p},
			&\mbox{if $\ell=1$,}\vspace*{4pt}
			\\
			\|g\|^{s/n}_{\B^s_{\alpha}}\,\,
			\|Q_g^{\tau,\ell-1}\|^{1/n}_{\alpha,p}\,\,
			\|Q^{\tau,\ell}_g\|^{1-1/n}_{\alpha,p}, &\mbox{if $\ell>1$.}
		\end{array}
		\right.
	\end{equation}
	Therefore~{\eqref{eqn:estimate:norm:Qf:Hp}}, \eqref{eqn:estimate:A:Hp} and~{\eqref{eqn:estimate:B:H^p}} imply that
	\begin{align*}
		\|Q^{\tau,1}_g\|_{\alpha,p}
		&\lesssim
		\|g\|^{s^2/m}_{\B^s_{\alpha}}\,\|Q^{\tau,1}_g\|^{1-s/m}_{\alpha,p}+
		\|g\|^{s/n}_{\B^s_{\alpha}}\,\,
		\|Q^{\tau,1}_g\|^{1-1/n}_{\alpha,p}\quad\mbox{and}
		\\
		\|Q^{\tau,\ell}_g\|_{\alpha,p}
		&\lesssim \|g\|^{s^2/m}_{\B^s_{\alpha}}
		\|Q^{\tau,\ell}_g\|^{1-s/(m\ell)}_{\alpha,p}
		+\|g\|^{s/n}_{\B^s_{\alpha}}
		\|Q_g^{\tau,\ell-1}\|^{1/n}_{\alpha,p}
		\|Q^{\tau,\ell}_g\|^{1-1/n}_{\alpha,p},
	\end{align*}
	for $\ell>1$. 
	Recall that $0<\|Q^{\tau,\ell}_g\|_{\alpha,p}<\infty$, for any $q$ and $\ell$, if $g$ is not constant, while $\|Q^{\tau,\ell}_g\|_{\alpha,p}=0$, for all $q$ and $\ell$, otherwise. In particular, if $g$ is constant then~{\eqref{eqn:upper:estimate:norm:Q:Hp}} holds. On the other hand, when $g$ is not constant, we may divide by $\|Q^{\tau,1}_g\|_{\alpha,p}$ and $\|Q^{\tau,\ell}_g\|_{\alpha,p}$ in the preceding estimates to get that
	\begin{equation*}
		1\lesssim
		\biggl(\frac{\|g\|^s_{\B^s_{\alpha}}}{\|Q^{\tau,1}_g\|_{\alpha,p}}\biggr)^{s/m}
		+\biggl(\frac{\|g\|^s_{\B^s_{\alpha}}}{\|Q^{\tau,1}_g\|_{\alpha,p}}\biggr)^{1/n}
	\end{equation*}
	and
	\begin{equation*}
		1\lesssim \biggl(\frac{\|g\|^s_{\B^s_{\alpha}}}{\|Q^{\tau,\ell}_g\|^{1/\ell}_{\alpha,p}}\biggr)^{s/m}
		+\biggl(\|g\|^s_{\B^s_{\alpha}}
		\frac{\|Q^{\tau,\ell-1}_g\|_{\alpha,p}}
		{\|Q^{\tau,\ell}_g\|_{\alpha,p}}\biggr)^{1/n}
		\quad(\ell>1).
	\end{equation*}
	Since $\frac{s}{m}=\frac{s}\tau\frac1{n}$, we may apply the convexity inequality 
	\[
	(x+y)^{n}\le 2^{n-1}(x^{n}+y^{n})
	\qquad(x,y>0),
	\] 
	to deduce that
	\begin{equation}\label{eqn:inequality:ell=1:Hp}
		1\lesssim
		\biggl(\frac{\|g\|^s_{\B^s_{\alpha}}}{\|Q^{\tau,1}_g\|_{\alpha,p}}\biggr)^{s/\tau}
		+\frac{\|g\|^s_{\B^s_{\alpha}}}{\|Q^{\tau,1}_g\|_{\alpha,p}}
	\end{equation}
	and
	\begin{equation}\label{eqn:induction:inequality:Hp}
		1\lesssim \Biggl(\frac{\|g\|^s_{\B^s_{\alpha}}}{\|Q^{\tau,\ell}_g\|^{1/\ell}_{\alpha,p}}\Biggr)^{s/\tau}
		+\|g\|^s_{\B^s_{\alpha}}
		\frac{\|Q^{\tau,\ell-1}_g\|_{\alpha,p}}
		{\|Q^{\tau,\ell}_g\|_{\alpha,p}}
		\quad(\ell>1).
	\end{equation}
	Now we can prove~{\eqref{eqn:upper:estimate:norm:Q:Hp}} by induction on $\ell$. First note that the case $\ell=1$ follows from~{\eqref{eqn:inequality:ell=1:Hp}}. 
	Let $\ell>1$.
	By the induction hypothesis,  
	\begin{equation*}
		\|Q_g^{\tau,\ell-1}\|_{\alpha,p}
		\lesssim\|g\|^{s(\ell-1)}_{\B^s_{\alpha}}. 
	\end{equation*} 
	Therefore, by~{\eqref{eqn:induction:inequality:Hp}},
	we have that
	\begin{equation*}
		1\lesssim \biggl(\frac{\|g\|^{s\ell}_{\B^s_{\alpha}}}{\|Q^{\tau,\ell}_g\|_{\alpha,p}}\biggr)^{s/(\tau\ell)}
		+
		\frac{\|g\|^{s\ell}_{\B^s_{\alpha}}}
		{\|Q^{\tau,\ell}_g\|_{\alpha,p}},
	\end{equation*}
	and so it follows that 
	$\|Q_g^{\tau,\ell}\|_{\alpha,p}\lesssim\|g\|^{s\ell}_{\B^s_{\alpha}}$. 
	Hence the proof is complete.
\end{proof}

\section{Proof of Theorem~{\ref{thm:sharp:estimate:norm:word} {\sf b)}} for $L=S^m_gT^n_g$: Necessary condition}

Taking into account Remark \ref{remark:SmTn:trivial:cases}, the necessity in Theorem~{\ref{thm:sharp:estimate:norm:word} {\sf b)} reduces to

\begin{theorem}\label{prop:sharp:lower:estimate:norm:word:ST:Hp}
	Let $m,n\in\N$, $n\ge2$, and $s=\frac{m}n+1$. Then
	\begin{equation*}
		\|g\|^{sn}_{\B^s_{\alpha}}
		\lesssim
		\opnorm{S^m_gT^n_g}_{\alpha,p}
		\qquad(g\in\H(\D)).
	\end{equation*}
\end{theorem}

 Theorem~{\ref{prop:sharp:lower:estimate:norm:word:ST:Hp}} is a direct consequence of the following two fundamental results:

\begin{proposition}\label{prop:sharp:lower:estimate:norm:word:1:Hp}
	Let $\ell\in\N$, $\tau\in\Q$, $\tau>0$, and $s=\tau+1$. Then
	\begin{equation*}
		\|g\|^{s\ell}_{\B^s_{\alpha}}\lesssim
		\opnorm{Q_g^{\tau,\ell}}_{\alpha,p}
		\qquad(g\in\H(\overline{\D})).
	\end{equation*}
\end{proposition}

\begin{proposition}\label{lem:sharp:lower:estimate:norm:word:2:Hp}
	Let  $\tau=\frac{m}n$, $m,n\in\N$. Then
	\begin{equation}
		\label{eqn:lem2:sharp:lower:estimate:norm:word1:Hp}
		\opnorm{Q_g^{\tau,n}}_{\alpha,p}
		\lesssim\opnorm{S^m_gT^n_g}_{\alpha,p}
		\qquad(g\in\H(\overline{\D})).
	\end{equation}
\end{proposition}
\vspace*{6pt}

We begin by proving Proposition~{\ref{lem:sharp:lower:estimate:norm:word:2:Hp}} because the proof of Proposition~{\ref{prop:sharp:lower:estimate:norm:word:1:Hp}} is more involved.

\begin{proof}[{\bf Proof of Proposition \ref{lem:sharp:lower:estimate:norm:word:2:Hp}}]
	Without loss of generality we may assume that $g\in\H(\overline{\D})$ is non-constant, since otherwise $\opnorm{S^m_gT^n_g}_{\alpha,p}=\opnorm{Q_g^{\tau,n}}_{\alpha,p}=0$. 
	
	Let $f\in A^p_{\alpha}(0)$ with $\|f\|_{\alpha,p}=1$. Then $|Q_g^{\tau,n}f|=|h|^{\frac1n}$, where $h=g^m\,T_g^nf$. Since $h$ belongs to $\H_0(\D)$ and satisfies that
	\[
	h'=m\,g^{m-1}g'\,T_g^nf+g^m(T_g^nf)'
	=m\,g^{m-1}g'\,T_g^nf+(S_g^mT_g^nf)',
	\]
\eqref{eqn:estimate:Apalpha:norm:derivative} gives that
	\begin{equation}
		\label{eqn:lem2:sharp:lower:estimate:norm:word2:Hp}
		\|Q_g^{\tau,n}f\|_{\alpha,p}
		\lesssim\|g^{m-1}g'\,T_g^nf\|_{T^p_{\alpha,2}}+\opnorm{S_g^mT_g^n}_{\alpha,p},
	\end{equation}
	because $\|(S_g^mT_g^nf)'\|_{T^p_{\alpha,2}}\lesssim\|S_g^mT_g^nf\|_{\alpha,p}\le\opnorm{S_g^mT_g^n}_{\alpha,p}$.
	Now, by~{\eqref{eqn:Mpq:simeq:BMOAs-norm}}, we have 
	\begin{equation*}
		\|g^{m-1}g'\,T_g^nf\|_{T^p_{\alpha,2}}
		=\|g^{m-1}\,T_g^nf\|_{T^p_{\alpha,2}(\nu_{g,1})}
		\lesssim\|T_g\|_{\alpha,p}\,\|g^{m-1}T_g^nf\|_{\alpha,p}.
	\end{equation*}
	
	If $m=1$, $\|g^{m-1}T_g^nf\|_{\alpha,p}=\|T_g^nf\|_{\alpha,p}\le\|T_g\|_{\alpha,p}^n$, so Proposition~{\ref{prop:boundedness:Tg}} shows that 
	$\|g^{m-1}g'\,T_g^nf\|_{T^p_{\alpha,2}}
	\lesssim\|T_g\|_{\alpha,p}^{n+1}
	\lesssim\opnorm{S^m_gT_g^n}_{\alpha,p}$, and, 
	by~{\eqref{eqn:lem2:sharp:lower:estimate:norm:word2:Hp}}, we obtain~{\eqref{eqn:lem2:sharp:lower:estimate:norm:word1:Hp}} for $m=1$. 
	
	Now assume that $m>1$. Then apply H\"older's inequality with exponents $m$ and $\frac{m}{m-1}$ to get that
	\begin{equation*}
		\|g^{m-1}T_g^nf\|_{\alpha,p}
		\le\|T_g^nf\|_{\alpha,p}^{\frac1m} \,\|Q^{\tau,n}_gf\|_{\alpha,p}^{\frac{m-1}m}
		\le\|T_g\|_{\alpha,p}^{\frac{n}m} \,\opnorm{Q^{\tau,n}_g}_{\alpha,p}^{\frac{m-1}m},
	\end{equation*}
	which, by Proposition~{\ref{prop:boundedness:Tg}}, implies that 
	\begin{equation*}
		\|g^{m-1}g'\,T_g^nf\|_{T^p_{\alpha,2}}
		\lesssim\|T_g\|_{\alpha,p}^{\frac{m+n}m} \,\opnorm{Q_g^{\tau,n}}_{\alpha,p}^{\frac{m-1}m}
		\lesssim\opnorm{S_g^mT_g^n}_{\alpha,p}^{\frac1m} \,\opnorm{Q_g^{\tau,n}}_{\alpha,p}^{1-\frac1m}.
	\end{equation*}
	This estimate and~{\eqref{eqn:lem2:sharp:lower:estimate:norm:word2:Hp}} show that
	\begin{equation*}
		\opnorm{Q_g^{\tau,n}}_{\alpha,p}	
		\lesssim\opnorm{S_g^mT_g^n}_{\alpha,p}^{\frac1m} \,\opnorm{Q^{\tau,n}_g}_{\alpha,p}^{1-\frac1m}
		+\opnorm{S_g^mT_g^n}_{\alpha,p}.
	\end{equation*}
	Since $0<\opnorm{Q^{\tau,n}_g}_{\alpha,p}<\infty$, we may divide by $\opnorm{Q^{\tau,n}_g}_{\alpha,p}$ to deduce that
	\begin{equation*}
		1\lesssim\biggr(\frac{\opnorm{S_g^mT_g^n}_{\alpha,p}}{\opnorm{Q^{\tau,n}_g}_{\alpha,p}}\biggr)^{\frac1m} + \frac{\opnorm{S_g^mT_g^n}_{\alpha,p}}{\opnorm{Q^{\tau,n}_g}_{\alpha,p}}.
	\end{equation*}
	Hence \eqref{eqn:lem2:sharp:lower:estimate:norm:word1:Hp} holds, and that ends the proof of the proposition.
\end{proof}

The proof of Proposition \ref{prop:sharp:lower:estimate:norm:word:1:Hp} is by induction on $\ell$. Since the proof is lengthy, we split it into the following two propositions which  will be proved in the next two sections. 

\begin{proposition}\label{lem:sharp:lower:estimate:norm:word:3:Hp}
	Let $\tau\in\Q$, $\tau>0$, and $s=\tau+1$. Then
	\begin{equation}
		\label{eqn:lem3:sharp:lower:estimate:norm:word:Hp}
		\|g\|^s_{\B^s_{\alpha}}\lesssim
		\opnorm{Q_g^{\tau,1}}_{\alpha,p} 
		\qquad(g\in\H(\overline{\D})).
	\end{equation}
\end{proposition}

\begin{proposition}\label{lem:sharp:lower:estimate:norm:word:4:Hp}
	Let  $\ell\in\N$, $\tau\in\Q$, $\tau>0$, and $s=\tau+1$. Assume that
	\begin{equation*}
		\|g\|^{s\ell}_{\B^s_{\alpha}}\lesssim
		\opnorm{Q^{\tau,\ell}_g}_{\alpha,p} 
		\qquad(g\in\H(\overline{\D})).
	\end{equation*}
	Then
	\begin{equation*}
		\|g\|^{s(\ell+1)}_{\B^s_{\alpha}}\lesssim
		\opnorm{Q_g^{\tau,\ell+1}}_{\alpha,p} 
		\qquad(g\in\H(\overline{\D})).
	\end{equation*}
\end{proposition}

\section{Proof of Proposition \ref{lem:sharp:lower:estimate:norm:word:3:Hp}}

The key tool to prove 
Proposition~{\ref{lem:sharp:lower:estimate:norm:word:3:Hp}}
is the following result.

\begin{proposition}
	\label{prop:key:tool:sharp:lower:estimate:norm:word:Hp}
	Let $\tau=\frac{m}n$, with $m,n\in\N$,
	and let $s=\tau+1$. Then, for every $\alpha\ge-1$ and $0<p<\infty$, there are two positive constants $C_{\alpha,p}$ \textup{(}only depending on $\alpha$ and $p$\textup{)} and $C_{\alpha,n,p}$ \textup{(}only depending on $\alpha$, $n$,  and  $p$\textup{)} which satisfy	   
	\begin{equation}
		\label{eqn:key:tool:sharp:lower:estimate:norm:word:Hp}	
		\|g\|^{2s}_{\B^s_{\alpha}}
		\le C_{\alpha,n,p}\,\opnorm{Q_g^{\tau,1}}_{\alpha,p}^2
		+ C_{\alpha,p}\tau\,\boldnorm{g}_{\alpha,\tau+\frac12}^{2\tau+1}\|g\|_{\B^1_{\alpha}}, 
	\end{equation}   

	for every $g\in\H(\overline{\D})$.
\end{proposition}	

In order to prove Proposition \ref{prop:key:tool:sharp:lower:estimate:norm:word:Hp} we need the following two lemmas.

\begin{lemma}\label{lem:Carleson:measure:subh:funct}
	For a nonnegative subharmonic function $v$ on $\D$, consider the measures $d\mu^{\alpha}=\rho^{\alpha+2}\,v\,dA$ and $d\widetilde{\mu}^{\alpha}(z)=|z|^2\,d\mu^{\alpha}(z)$, for $\alpha\ge-1$. Then there is a constant $0<c_{\alpha}<1$, independent of $v$, such that
	\begin{equation}
		\label{eqn:lem:Carleson:measure:subh:funct0}  
		c_{\alpha}\,\|\mu^{\alpha}\|_{\mathcal{C}(\alpha)}
		\le\|\widetilde{\mu}^{\alpha}\|_{\mathcal{C}(\alpha)}\le
		\|\mu^{\alpha}\|_{\mathcal{C}(\alpha)}.
	\end{equation}
\end{lemma}

\begin{proof}
	Since $\widetilde{\mu}^{\alpha}(S(a))\le\mu^{\alpha}(S(a))$, for every  $a\in\D$, we have that $\|\widetilde{\mu}^{\alpha}\|_{\mathcal{C}(\alpha)}\le\|\mu^{\alpha}\|_{\mathcal{C}(\alpha)}$. Then, by~{\eqref{eqn:alpha+2-Carleson:measures:equiv}}, in order to complete the proof of~{\eqref{eqn:lem:Carleson:measure:subh:funct0}}, it is enough to prove that 
	\begin{equation}
		\label{eqn:lem:Carleson:measure:subh:funct1}
		\int_{\D}\frac{d\mu^{\alpha}(z)}{|1-\overline{\lambda}z|^{2\alpha+4}}
		\le  C_{\alpha}\, 
		\int_{\D}\frac{d\widetilde{\mu}^{\alpha}(z)}{|1-\overline{\lambda}z|^{2\alpha+4}}
		\qquad(\lambda\in\D),
	\end{equation} 
where $C_{\alpha}>1$ is a constant only depending on $\alpha$.

	First, let $D_r:=D(0,r)$, for any $r>0$, and note that
	\begin{equation*}
		\int_{D_{\frac14}}\frac{d\mu^{\alpha}(z)}{|1-\overline{\lambda}z|^{2\alpha+4}}
		\le 4^{\alpha+2}\mu^{\alpha}(D_{\frac14})
		\,\,\,\mbox{ and }\,\,\, 
		\widetilde{\mu}^{\alpha}(D_{\frac12})
		\le 4^{\alpha+2}\int_{D_{\frac12}}
		\frac{d\widetilde{\mu}^{\alpha}(z)}{|1-\overline{\lambda}z|^{2\alpha+4}}.
	\end{equation*}
	Now the subharmonicity of $v$ and the inequality $\rho(r^2)\le2\rho(r)$, for $0<r<1$, show that
	\begin{align*}
		\widetilde{\mu}^{\alpha}(D_{\frac12})
		&=\int_0^{\frac12}\rho(r)^{\alpha+2}
		\biggl(\int_{\T}v_r\,d\sigma\biggr)r^2dr^2 \\
		&\ge \frac1{2^{\alpha+2}}\int_0^{\frac12}\rho(r^2)^{\alpha+2}
		\biggl(\int_{\T}v_{r^2}\,d\sigma\biggr)r^2dr^2 
		=\frac1{2^{\alpha+4}}\,\mu^{\alpha}(D_{\frac14}),
	\end{align*}
	and so
	\begin{equation}
		\label{eqn:lem:Carleson:measure:subh:funct2} 	
		\int_{D_{\frac14}}\frac{d\mu^{\alpha}(z)}{|1-\overline{\lambda}z|^{2\alpha+4}}
		\le 2^{5\alpha+12}\int_{D_{\frac12}}
		\frac{d\widetilde{\mu}^{\alpha}(z)}{|1-\overline{\lambda}z|^{2\alpha+4}}
		\qquad(a\in\D).
	\end{equation}
	Morevover, it is clear that
	\begin{equation}
		\label{eqn:lem:Carleson:measure:subh:funct3} 	
		\int_{\D\setminus D_{\frac14}}
		\frac{d\mu^{\alpha}(z)}{|1-\overline{\lambda}z|^{2\alpha+4}}
		\le 16\,\int_{\D\setminus D_{\frac14}}
		\frac{d\widetilde{\mu}^{\alpha}(z)}{|1-\overline{\lambda}z|^{2\alpha+4}}
		\qquad(a\in\D).
	\end{equation} 
	Finally, \eqref{eqn:lem:Carleson:measure:subh:funct1} directly follows from \eqref{eqn:lem:Carleson:measure:subh:funct2} and 
	\eqref{eqn:lem:Carleson:measure:subh:funct3}.   
\end{proof}

\begin{proof}[{\bf Proof of Proposition \ref{prop:key:tool:sharp:lower:estimate:norm:word:Hp}}]
	Along this proof the finite positive constants will be denoted by  $C_{\alpha,p},\,C'_{\alpha,p},\,C''_{\alpha,p}$ (constants only depending on $\alpha$ and $p$), and $C_{\alpha,n,p}$ (constants only depending on $\alpha$, $n$, and $p$). The values of those constants may change from line to line.	
	
	Let $g\in\H(\overline{\D})$ and  $f\in A^p_{\alpha}(0)$, with $\|f\|_{\alpha,p}=1$. Note that
	\[
	\|Q^{\tau,1}_gf\|_{\alpha,p}=\|h\|^{1/n}_{\alpha,\frac{p}n},
	\quad\mbox{where $h=g^m(T_gf)^n\in A^{p/n}_{\alpha}(0)$.}
	\]
	Then estimate \eqref{eqn:Calderon:area:thm:3} for $q=\frac1n$ gives that
	\begin{equation}\label{eqn:Calderon:area:thm}
		\||h|^{\frac2n-2}|h'|^2\|_{T^{\frac{p}2}_{\alpha,1}}
		\le C_{\alpha,n,p}\,\|Q^{\tau,1}_gf\|_{\alpha,p}^2
		\le C_{\alpha,n,p}\,\opnorm{Q^{\tau,1}_g}_{\alpha,p}^2.
	\end{equation}
	Now since 
	\begin{equation*}
		|h|^{\frac2n-2}=|g|^{2\tau-2m}|T_gf|^{2-2n}
		\mbox{ and }
		|h'|^2=|g|^{2m-2}|g'|^2|T_gf|^{2n-2}|mT_gf+ngf|^2,
	\end{equation*}
	we have that
	\begin{equation*}
		\tfrac1{n^2}|h|^{\frac2n-2}|h'|^2
		= |g|^{2\tau-2} |g'|^2 |\tau T_gf+gf|^2
		\ge |g|^{2\tau-2} |g'|^2 (\tau|T_gf|-|gf|)^2.
	\end{equation*}
	But $(\tau|T_gf|-|gf|)^2\ge |gf|^2-2\tau |gfT_gf|$, and so
	we get that
	\begin{equation*}
		\tfrac1{n^2}|h|^{\frac2n-2}|h'|^2
		\ge |g|^{2\tau}|g'|^2|f|^2
		-2\tau|g|^{2\tau-1}|g'|^2|fT_gf|,
	\end{equation*}	
	or equivalently
	\begin{equation*}
		|g|^{2\tau}|g'|^2|f|^2
		\le \tfrac1{n^2}|h|^{\frac2n-2}|h'|^2
		+2\tau|g|^{2\tau-1}|g'|^2|fT_gf|.
	\end{equation*} 
	Therefore, by \eqref{eqn:Calderon:area:thm}, we have
	\begin{equation*}
		\label{eqn:lem3:sharp:lower:estimate:norm:word:Hp:1}	
		\||g|^{2\tau}|g'|^2|f|^2\|_{T^{\frac{p}2}_{\alpha,1}}
		\le C_{\alpha,n,p}\,\opnorm{Q^{\tau,1}_g}_{\alpha,p}^2
		+C_{\alpha,p}\tau\,\||g|^{2\tau-1}|g'|^2|fT_gf|\|_{T^{\frac{p}2}_{\alpha,1}}.
	\end{equation*}
	But $\||g|^{2\tau-1}|g'|^2|fT_gf|\|_{T^{\frac{p}2}_{\alpha,1}}
	=\|fT_gf\|_{T^{\frac{p}2}_{\alpha,1}(\nu_{g,\tau+\frac12})}$, so  \eqref{eqn:Mpq:simeq:BMOAs-norm} shows  that
\begin{equation*}
		\||g|^{2\tau-1}|g'|^2|fT_gf|\|_{T^{\frac{p}2}_{\alpha,1}}
		\le C_{\alpha,p}\,\boldnorm{g}_{\alpha,\tau+\frac12}^{2\tau+1}
       \|fT_gf\|_{\alpha,\frac{p}2}.
	\end{equation*}
	Moreover, Schwarz's inequality gives that
	\begin{equation*}
		\|f(T_gf)\|_{\alpha,\frac{p}2}
		\le\|f\|_{\alpha,p}\|T_gf\|_{\alpha,p}\le C_{\alpha,p}\,\|g\|_{\B^1_{\alpha}},
	\end{equation*}
	and so
\begin{equation*}
		\||g|^{2\tau-1}|g'|^2|fT_gf|\|_{T^{\frac{p}2}_{\alpha,1}}
		\le C_{\alpha,p}\,
	\boldnorm{g}_{\alpha,\tau+\frac12}^{2\tau+1}\|g\|_{\B^1_{\alpha}}.
	\end{equation*}
	Since
	$\||g|^{2\tau}|g'|^2|f|^2\|_{T^{\frac{p}2}_{\alpha,1}}=\|f\|^2_{T^p_{\alpha,2}(\nu_{g,s})}$, it follows that
\begin{equation*}
		\|f\|^2_{T^p_{\alpha,2}(\nu_{g,s})}\le C_{\alpha,n,p}\,\opnorm{Q^{\tau,1}_g}_{\alpha,p}^2+C_{\alpha,p}\tau\boldnorm{g}_{\alpha,\tau+\frac12}^{2\tau+1}\|g\|_{\B^1_{\alpha}}.
	\end{equation*}
	By taking supremum on $f\in A^p_{\alpha}(0)$, $\|f\|_{\alpha,p}=1$, we deduce that
\begin{equation}
		\label{eqn:key:tool:sharp:lower:estimate:norm:word:Hp:1}
		M_{p,2}(\widetilde{\nu}_{g,s},\alpha)\le C_{\alpha,n,p}\,\opnorm{Q^{\tau,1}_g}_{\alpha,p}^2
		+C_{\alpha,p}\tau\boldnorm{g}_{\alpha,\tau+\frac12}^{2\tau+1}\|g\|_{\B^1_{\alpha}},
\end{equation}
	where
	$d\widetilde{\nu}_{g,s}(z)=|z|^2\,d\nu_{g,s}(z)$ and $M_{p,2}(\widetilde{\nu}_{g,s},\alpha)$ is defined by~{\eqref{eqn:definition:Mpqnu}}.
	By Theorem~{\ref{thm:Carleson:measure:norm:tent:spaces}}, Lemmas~{\ref{lem:subharmonic1}} and~{\ref{lem:Carleson:measure:subh:funct}}, estimates \eqref{eqn:derivative:symbol:Carleson:measure} and \eqref{eqn:equiv:norms:Bq}, and the identity \eqref{eqn:BMOAqnorm:equal:boldnorm:-1},
	 we have that
	\begin{equation}
		\label{eqn:key:tool:sharp:lower:estimate:norm:word:Hp:2}		
		M_{p,2}(\widetilde{\nu}_{g,s},\alpha)
		\ge C_{\alpha,p}\,
		\|\widetilde{\mu}_{g,s}\|_{\mathcal{C}(\alpha)}
		\ge C'_{\alpha,p}\,\|\mu_{g,s}\|_{\mathcal{C}(\alpha)}
		\ge C''_{\alpha,p}\,\|g\|^{2s}_{\B^s_{\alpha}}.
	\end{equation}
	Finally, it is clear that \eqref{eqn:key:tool:sharp:lower:estimate:norm:word:Hp:1} and \eqref{eqn:key:tool:sharp:lower:estimate:norm:word:Hp:2} give	
	\eqref{eqn:key:tool:sharp:lower:estimate:norm:word:Hp}, and that ends the proof.
\end{proof}	

Now we prove that the estimate \eqref{eqn:lem3:sharp:lower:estimate:norm:word:Hp} holds for $\tau>0$ small enough.

\begin{proposition}\label{prop:sharp:lower:estimate:norm:word:tau:small}
	For any $\alpha\ge-1$ and $0<p<\infty$ there is a constant $\tau_{\alpha,p}>0$ such that
	the estimate \eqref{eqn:lem3:sharp:lower:estimate:norm:word:Hp}
	holds for any $\tau\in\Q$ with $0<\tau<\tau_{\alpha,p}$.
\end{proposition}

\begin{proof}
By~{\eqref{eqn:key:estimate:Carleson:norms:symbols0}}, \eqref{eqn:BMOAqnorm:equal:boldnorm:-1},  \eqref{eqn:equiv:norms:Bq}, \eqref{eqn:key:estimate:Carleson:norms:symbols1} and~{\eqref{eqn:key:estimate:Carleson:norms:symbols2}}
	 we have that there is an absolute constant $C>0$ such that
	\begin{equation*}
		\boldnorm{g}_{\alpha,\tau+\frac12}^{2\tau+1}\|g\|_{\B^1_{\alpha}}
		\le C\|g\|^{2s}_{\B^s_{\alpha}},\quad\mbox{for any $\tau>0$.}
	\end{equation*} 
	 Then Proposition~{\ref{prop:key:tool:sharp:lower:estimate:norm:word:Hp}} shows that if $\tau=\frac{m}n$, where $m,n\in\N$,
	then there exist two positive constants $C_{\alpha,p}>0$ \textup{(}only depending on $\alpha$ and $p$\textup{)} and $C_{\alpha,n,p}$ \textup{(}only depending on $\alpha$, $n$, and $p$\textup{)} which satisfy
	\begin{equation*}
		\|g\|^{2s}_{\B^s_{\alpha}}
		\le C_{\alpha,n,p}\,\opnorm{Q_g^{\tau,1}}_p^2
		+ C_{\alpha,p}\tau\,\|g\|^{2s}_{\B^s_{\alpha}}
		\qquad(g\in\H(\overline{\D})). 
	\end{equation*}	   
	Therefore it is clear that the estimate \eqref{eqn:lem3:sharp:lower:estimate:norm:word:Hp} holds for every $\tau\in\Q$ such that $0<\tau<\tau_{\alpha,p}=1/C_{\alpha,p}$, and that ends the proof.
\end{proof}

As a consequence of Proposition~{\ref{prop:sharp:lower:estimate:norm:word:tau:small}} we obtain the following weak version of 
estimate~{\eqref{eqn:lem3:sharp:lower:estimate:norm:word:Hp}},
where  $\|g\|^{s}_{\B^s_{\alpha}}$ is replaced by $\|g\|^{s}_{\B^1_{\alpha}}$:

\begin{proposition}
	\label{eqn:prop:lower:estimate:Qtau1:BMOA}
	Let $\tau>0$, and $s=\tau+1$. Then
	\begin{equation}
		\label{eqn:prop:lower:estimate:Qtau1:BMOA}
		\|g\|^s_{\B^1_{\alpha}}\lesssim
		\opnorm{Q_g^{\tau,1}}_{\alpha,p} 
		\qquad(g\in\H(\overline{\D})).
	\end{equation}
\end{proposition}

\begin{proof}
	First note that if~{\eqref{eqn:prop:lower:estimate:Qtau1:BMOA}} holds for some $\tau=\tau_0>0$, then it also holds for any $\tau>\tau_0$. This is so because, for any $0<\tau_0<\tau$, we have the estimate
	\begin{equation}
		\label{eqn:prop:lower:estimate:Qtau1:BMOA:1}	
		\opnorm{Q_g^{\tau_0,1}}_{\alpha,p}^{\frac1{\tau_0}}
		\lesssim	\opnorm{Q_g^{\tau,1}}_{\alpha,p}^{\frac1{\tau}}\, \|g\|^{\frac1{\tau_0}-\frac1{\tau}}_{\B^1_{\alpha}}
		\qquad(g\in\H(\overline{\D})).
	\end{equation} 
	Indeed, just apply H\"{o}lder's inequality to get that 
	\begin{equation*}
		\|Q_g^{\tau_0,1}f\|_{\alpha,p}^{\frac1{\tau_0}}
		\le	\|Q_g^{\tau,1}f\|_{\alpha,p}^{\frac1{\tau}} \|T_gf\|^{\frac1{\tau_0}-\frac1{\tau}}_{\alpha,p}
		\le	\opnorm{Q_g^{\tau,1}}_{\alpha,p}^{\frac1{\tau}}\, 
		\|T_g\|^{\frac1{\tau_0}-\frac1{\tau}}_{\alpha,p}\|f\|_{\alpha,p},
	\end{equation*}
for every $f\in A^p_{\alpha}(0)$,
	from which~{\eqref{eqn:prop:lower:estimate:Qtau1:BMOA:1}} follows: 
	\begin{equation*}
		\opnorm{Q_g^{\tau_0,1}}^{\frac1{\tau_0}}_{\alpha,p}
		\le	\opnorm{Q_g^{\tau,1}}_{\alpha,p}^{\frac1{\tau}}\, 
		\|T_g\|^{\frac1{\tau_0}-\frac1{\tau}}_{\alpha,p}
		\lesssim\opnorm{Q_g^{\tau,1}}_{\alpha,p}^{\frac1{\tau}}\,\|g\|_{\B^1_{\alpha}}^{\frac1{\tau_0}-\frac1{\tau}}\quad(g\in\H(\overline{\D})).
	\end{equation*}
	
Now Proposition~{\ref{prop:sharp:lower:estimate:norm:word:tau:small}} and estimates \eqref{eqn:key:estimate:Carleson:norms:symbols1} and  \eqref{eqn:key:estimate:Carleson:norms:symbols2} show that~{\eqref{eqn:prop:lower:estimate:Qtau1:BMOA}} hold for any positive $\tau_0\in\Q$ which is small enough. Hence~{\eqref{eqn:prop:lower:estimate:Qtau1:BMOA}} must hold for any positive $\tau\in\R$.
\end{proof}

\begin{proof}[{\bf Proof of Proposition \ref{lem:sharp:lower:estimate:norm:word:3:Hp}}]
	Proposition~{\ref{prop:key:tool:sharp:lower:estimate:norm:word:Hp}} and~{\eqref{eqn:prop:lower:estimate:Qtau1:BMOA}} together with~{\eqref{eqn:key:estimate:Carleson:norms:symbols1}} give the estimate
	\begin{equation*}
		\|g\|^{2s}_{\B^s_{\alpha}}
		\lesssim \opnorm{Q_g^{\tau,1}}^2_{\alpha,p}
		+\|g\|_{\B^s_{\alpha}}^{2\tau+1}\opnorm{Q_g^{\tau,1}}^{\frac1s}_{\alpha,p}
		\qquad(g\in\H(\overline{\D})).
	\end{equation*}
	It follows that
	\begin{equation*}
		1\lesssim
		\frac{\opnorm{Q_g^{\tau,1}}^2_{\alpha,p}}{\|g\|^{2s}_{\B^s_{\alpha}}}
		+\left(\frac{\opnorm{Q_g^{\tau,1}}^2_{\alpha,p}}{\|g\|^{2s}_{\B^s_{\alpha}}}\right)^{\frac1{2s}}
		\quad(g\in\H(\overline{\D}),\,g\mbox{ non-constant}),
	\end{equation*}
	which clearly shows that
	\eqref{eqn:lem3:sharp:lower:estimate:norm:word:Hp} holds, and that ends the proof.
\end{proof}

\section{Proof of Proposition \ref{lem:sharp:lower:estimate:norm:word:4:Hp}}

Proposition \ref{lem:sharp:lower:estimate:norm:word:4:Hp} is a direct consequence of the following estimate.
\begin{proposition}
	\label{prop:key:tool:lower:estimate:norm:word}
	Let $\ell\in\N$, $\tau\in\Q$, $\tau>0$, and $s=\tau+1$. Write $\tau$ as $\tau=\frac{m}n$, where $m,n\in\N$ and $n>1+\frac{2s}{\tau}$.
	Then we have that
	\begin{equation}
		\label{eqn:key:tool:lower:estimate:norm:word}
		\opnorm{Q^{\tau,\ell}_g}_{\alpha,p}^n
		\lesssim\opnorm{Q_g^{\tau,\ell+1}}_{\alpha,p}\,
		\|g\|^{s\ell(n-\frac{\ell+1}{\ell})}_{\B^s_{\alpha}}
		\qquad(g\in\H(\overline{\D})).
	\end{equation}
\end{proposition}	

In order to prove Proposition~{\ref{prop:key:tool:lower:estimate:norm:word}}
we need the following technical lemma.

\begin{lemma}
	\label{lem:key:tool:lower:estimate:norm:word}
	Let $\ell$, $\tau=\frac{m}n$, and $s$ as in the statement of Proposition~{\ref{prop:key:tool:lower:estimate:norm:word}}.
	Let $g\in\H(\overline{\D})$ and $f\in A^p_{\alpha}(0)$ such that $\|f\|_{\alpha,p}=1$. For  $k=0,1,2$, let 
	\begin{equation}
		\label{eqn:def:Fk}
		F_k=(T_g^{\ell+1}f)\,g^{m\ell+k-2}g'\,
		(T_g^{\ell}f)^{n-1-k}\,  (T_g^{\ell-1}f)^k,
	\end{equation}
where, as usual, $T^0_gf=f$.
	Then there is a constant $C>0$, which does not depend on $f$ and $g$, so that
	\begin{equation*}
		\|F_k\|_{T^{p/n}_{\alpha,2}}
		\le C\,\opnorm{Q_g^{\tau,\ell+1}}_{\alpha,p}\,
		\|g\|^{s\ell(n-\frac{\ell+1}{\ell})}_{\B^s_{\alpha}}.
	\end{equation*}
\end{lemma}

\begin{proof}
	Along this proof, $\nu$ is the measure $\nu_{g,s}$ defined by~{\eqref{eqn:def:nugs}}. Moreover, all the constants associated to $\lesssim$ do not depend on $f$ and $g$. 	
	
	Let $h_j=g^{mj}(T^j_gf)^n$, for any $j\in\N_0$. Then 
	$h_j\in\H(\D)$ and 
	\begin{equation}
		\label{eqn:h_j:Q}	
		\|h_j\|^{1/n}_{\alpha,p/n}=\|Q^{\tau,j}_gf\|_{\alpha,p}\le\opnorm{Q^{\tau,j}_g}_{\alpha,p}.
	\end{equation}	
	Moreover,
	\begin{align*}
		\|F_k\|_{T^{p/n}_{\alpha,2}}
		&=\|\,|T_g^{\ell+1}f|\,|g|^{n\tau\ell+k-2}\,|g'|\,
		|T_g^{\ell}f|^{n-1-k}\, |T_g^{\ell-1}f|^k\,\|_{T^{p/n}_{\alpha,2}}
		\\
		&=\|(|g|^{\tau}|g'|)\,
		|h_{\ell+1}|^{1/n}\,|h_{\ell}|^{a_k/n}\,
		|h_{\ell-1}|^{k/n}\,|T_g^{\ell}f|^{b_k}\,\|_{T^{p/n}_{\alpha,2}}
		\\
		&=\|\,|h_{\ell+1}|^{1/n}\,|h_{\ell}|^{a_k/n}\,
		|h_{\ell-1}|^{k/n}\,|T_g^{\ell}f|^{b_k}\,
		\|_{T^{p/n}_{\alpha,2}(\nu)},
	\end{align*}	
	where $a_k=n-k-1-\frac{(2-k)s}{\tau\ell}$ and $b_k=\frac{(2-k)s}{\tau\ell}$, that is,
	\begin{align*}
		\|F_0\|_{T^{p/n}_{\alpha,2}}
		&=\|\,|h_{\ell+1}|^{1/n}\,|h_{\ell}|^{a_0/n}\,
		|T_g^{\ell}f|^{b_0}\,
		\|_{T^{p/n}_{\alpha,2}(\nu)}
		\\
		\|F_1\|_{T^{p/n}_{\alpha,2}}
		&=\|\,|h_{\ell+1}|^{1/n}\,|h_{\ell}|^{a_1/n}\,
		|h_{\ell-1}|^{1/n}\,|T_g^{\ell}f|^{b_1}\,
		\|_{T^{p/n}_{\alpha,2}(\nu)}
		\\
		\|F_2\|_{T^{p/n}_{\alpha,2}}
		&=\|\,|h_{\ell+1}|^{1/n}\,|h_{\ell}|^{a_2/n}\,
		|h_{\ell-1}|^{2/n}\,
		\|_{T^{p/n}_{\alpha,2}(\nu)}.
	\end{align*}
	Note that $b_0,b_1>0$, $a_k>0$ (since $n>1+\frac{2s}{\tau}$) and $a_k+b_k=n-k-1$, for $k=0,1,2$.
	So we may apply Proposition~{\ref{prop:Holder:ineq:tent:spaces}},
	with the choice of exponents
	\[
	\left\{\begin{array}{l}
		\mbox{$(p_1,p_2,p_3)=(p,p/a_0,p/b_0)$, and $(q_1,q_2,q_3)=(6,6,6)$, if $k=0$,}
		\\
		\mbox{$(p_1,p_2,p_3,p_4)=(p,p/a_1,p,p/b_1)$, and $(q_1,q_2,q_3,q_4)=(8,8,8,8)$, if $k=1$,}
		\\
		\mbox{$(p_1,p_2,p_3)=(p,p/a_2,p/2)$, and $(q_1,q_2,q_3)=(6,6,6)$, if $k=2$,}
	\end{array}\right.
	\]
	and the identity~{\eqref{eqn:Tpq:norm:of:a:power}} to obtain that
	\begin{align*}
		\|F_0\|_{T^{p/n}_{\alpha,2}}
		&\le
		\|h_{\ell+1}\|_{T^{p/n}_{\alpha,6/n}(\nu)}^{1/n}\,
		\|h_{\ell}\|_{T^{p/n}_{\alpha,6a_0/n}(\nu)}^{a_0/n}\,
		\|T_g^{\ell}f\|_{T^{p}_{\alpha,6b_0}(\nu)}^{b_0}
		\\
		\|F_1\|_{T^{p/n}_{\alpha,2}}
		&\le\|h_{\ell+1}\|_{T^{p/n}_{\alpha,8/n}(\nu)}^{1/n}
		\|h_{\ell}\|_{T^{p/n}_{\alpha,8a_1/n}(\nu)}^{a_1/n}
		\|h_{\ell-1}\|_{T^{p/n}_{\alpha,8/n}(\nu)}^{1/n}
		\|T_g^{\ell}f\|_{T^{p}_{\alpha,8b_1}(\nu)}^{b_1}
		\\
		\|F_2\|_{T^{p/n}_{\alpha,2}}
		&\le\|h_{\ell+1}\|_{T^{p/n}_{\alpha,6/n}(\nu)}^{1/n}
		\|h_{\ell}\|_{T^{p/n}_{\alpha,6a_2/n}(\nu)}^{a_2/n}
		\|h_{\ell-1}\|_{T^{p/n}_{\alpha,12/n}(\nu)}^{2/n}
	\end{align*}
	
	Finally, the estimates~{\eqref{eqn:Mpq:simeq:BMOAs-norm}}, \eqref{eqn:h_j:Q}, \eqref{eqn:upper:estimate:norm:Q:Hp}, \eqref{eqn:key:estimate:Carleson:norms:symbols1} and~{\eqref{eqn:key:estimate:Carleson:norms:symbols2}} complete the proof.
	Namely, for $k=1$ we have that
	\begin{align*}
		\|F_1\|_{T^{p/n}_{\alpha,2}}
		&\lesssim 
		\|g\|^s_{\B^s_{\alpha}}\,
		\|Q^{\tau,\ell+1}_gf\|_{\alpha,p}\,
		\|Q^{\tau,\ell}_gf\|_{\alpha,p}^{a_1}\,
		\|Q^{\tau,\ell-1}_gf\|_{\alpha,p}\,
		\|T_g^{\ell}f\|_{\alpha,p}^{b_1}
		\\
		&\lesssim 
		\|g\|^s_{\B^s_{\alpha}}\,
		\opnorm{Q_g^{\tau,\ell+1}}_{\alpha,p}\, 
		\opnorm{Q_g^{\tau,\ell}}^{a_1}_{\alpha,p}\,
		\opnorm{Q_g^{\tau,\ell-1}}_{\alpha,p}\,
		\|g\|^{\ell b_1}_{\B^1_{\alpha}}
		\\
		&\lesssim
		\opnorm{Q_g^{\tau,\ell+1}}_{\alpha,p}\, 
		\|g\|^{s\ell a_1+s\ell+\ell b_1}_{\B^s_{\alpha}}
		=\opnorm{Q_g^{\tau,\ell+1}}_{\alpha,p}\,
		\|g\|^{s\ell(n-\frac{\ell+1}{\ell})}_{\B^s_{\alpha}}.
	\end{align*}
	Similarly, we also obtain that
	\begin{align*}
		\|F_0\|_{T^{p/n}_{\alpha,2}}
		&\lesssim
		\opnorm{Q_g^{\tau,\ell+1}}_{\alpha,p}\, 
		\|g\|^{s\ell a_0+\ell b_0+s}_{\B^s_{\alpha}}
		=\opnorm{Q_g^{\tau,\ell+1}}_{\alpha,p}\,
		\|g\|^{s\ell(n-\frac{\ell+1}{\ell})}_{\B^s_{\alpha}}
		\\
		\|F_2\|_{T^{p/n}_{\alpha,2}}
		&\lesssim 
		\opnorm{Q_g^{\tau,\ell+1}}_{\alpha,p}\, 
		\|g\|^{s\ell a_2+2s\ell-s}_{\B^s_{\alpha}}
		=\opnorm{Q_g^{\tau,\ell+1}}_{\alpha,p}\,
		\|g\|^{s\ell(n-\frac{\ell+1}{\ell})}_{\B^s_{\alpha}},
	\end{align*}
	and that ends the proof of the lemma.
\end{proof}	

\begin{proof}[{\bf Proof of Proposition~{\ref{prop:key:tool:lower:estimate:norm:word}}}]
	Along this proof  all the constants associated to $\lesssim$ do not depend on $f$ and $g$. 
	Let $f\in A^p_{\alpha}(0)$ such that $\|f\|_{\alpha,p}=1$. Then, as in the proof of the previous lemma, for any $j\in\N_0$ we consider the function $h_j:=g^{mj}\,(T_g^jf)^n\in\H_0(\D)$ which satisfies \eqref{eqn:h_j:Q}.
	Now 
	\begin{align*}
		h'_{\ell}&=m\ell\, g^{m\ell-1} g'(T_g^{\ell}f)^n
		+ng^{m\ell} g' (T_g^{\ell}f)^{n-1}\, T_g^{\ell-1}f\\
		&=m\ell\,(T_g^{\ell+1}f)' g^{m\ell-1}(T_g^{\ell}f)^{n-1}
		+n(T_g^{\ell+1}f)' g^{m\ell}  (T_g^{\ell}f)^{n-2}\, T_g^{\ell-1}f,
	\end{align*}
	so~{\eqref{eqn:estimate:Apalpha:norm:derivative}} shows that $\|Q^{\tau,\ell}_gf\|_{\alpha,p}^n\lesssim\|h'_{\ell}\|_{T^{p/n}_{\alpha,2}}\lesssim A+B$, where 
	\begin{equation*}
		A=\|(T_g^{\ell+1}f)' g^{m\ell-1}(T_g^{\ell}f)^{n-1}   \|_{T^{\frac{p}n}_{\alpha,2}}
		\mbox{ and }
		B=\|(T_g^{\ell+1}f)' g^{m\ell}  (T_g^{\ell}f)^{n-2}\, T_g^{\ell-1}f\|_{T^{\frac{p}n}_{\alpha,2}}.
	\end{equation*}
	Then~{\eqref{eqn:key:tool:lower:estimate:norm:word}} will be proved once we show the following two estimates:
	\begin{align}
		\label{eqn:estimate:A}
		A&\lesssim \opnorm{Q_g^{\tau,\ell+1}}_{\alpha,p}\,
		\|g\|^{s\ell(n-\frac{\ell+1}{\ell})}_{\B^s_{\alpha}} 
		\\
		\label{eqn:estimate:B}	
		B&\lesssim \opnorm{Q_g^{\tau,\ell+1}}_{\alpha,p}\,
		\|g\|^{s\ell(n-\frac{\ell+1}{\ell})}_{\B^s_{\alpha}}.	
	\end{align}
	
	\noindent{\sf Estimate of $A$:}
	\begin{align*}
		A
		&\lesssim
		\|\bigl((T_g^{\ell+1}f)
		g^{m\ell-1}(T_g^{\ell}f)^{n-1}\bigr)'\|_{T^{\frac{p}n}_{\alpha,2}}
		+\|(T_g^{\ell+1}f)\bigl(g^{m\ell-1}(T_g^{\ell}f)^{n-1}    \bigr)'\|_{T^{\frac{p}n}_{\alpha,2}} 
		\\
		&\lesssim A_1+\|F_0\|_{T^{\frac{p}n}_{\alpha,2}} +\|F_1\|_{T^{\frac{p}n}_{\alpha,2}},
	\end{align*}
	where $A_1=\|(T_g^{\ell+1}f) g^{m\ell-1}(T_g^{\ell}f)^{n-1} 
	\|_{\alpha,\frac{p}n}$, and $F_0$, $F_1$ are defined by \eqref{eqn:def:Fk}.
	Note that 
	\begin{equation*}
		A_1=
		\||T_g^{\ell+1}f| |g|^{n\tau\ell-1}|T_g^{\ell}f|^{n-1} 
		\|_{\alpha,\frac{p}n}
		=\||Q_g^{\tau,\ell+1}f|\,|Q_g^{\tau,\ell}f|^a|T_g^{\ell}f|^b\|_{\alpha,\frac{p}n},
	\end{equation*} 
	where  $a=n-1-\frac{s}{\tau\ell}>0$ (since $n>1+\frac{2s}{\tau}$) and $b=\frac{s}{\tau\ell}>0$. 
	So, by applying H\"{o}lder's inequality with exponents $p_1=n$, $p_2=\frac{n}a$ and $p_3=\frac{n}b$, and  estimates
	\eqref{eqn:upper:estimate:norm:Q:Hp}, \eqref{eqn:key:estimate:Carleson:norms:symbols1} and \eqref{eqn:key:estimate:Carleson:norms:symbols2}, we get  that
	\begin{align*}
		A_1
		&\le \|Q_g^{\tau,\ell+1}f\|_{\alpha,p}\, 
		\|Q_g^{\tau,\ell}f\|^a_{\alpha,p}\,
		\|T_g^{\ell}f\|^b_{\alpha,p} 
		\le \opnorm{Q_g^{\tau,\ell+1}}_{\alpha,p}\, 
		\opnorm{Q_g^{\tau,\ell}}^a_{\alpha,p}\,
		\|T_g\|^{\ell b}_{\alpha,p}
		\\
		&
		\lesssim\opnorm{Q_g^{\tau,\ell+1}}_{\alpha,p}\,
		\|g\|^{s\ell a}_{\B^s_{\alpha}}\,\|g\|^{\ell b}_{\B^1_{\alpha}}
		\lesssim 
		\opnorm{Q_g^{\tau,\ell+1}}_{\alpha,p}\,
		\|g\|^{s\ell(n-\frac{\ell+1}{\ell})}_{\B^s_{\alpha}},    
	\end{align*}
	since $s\ell a+\ell b=s\ell(n-1)+(1-s)\tfrac{s}{\tau}=s\ell(n-1)-s
	=s\ell(n-\tfrac{\ell+1}{\ell})$. 
	Moreover, Lemma \ref{lem:key:tool:lower:estimate:norm:word} shows that $\|F_0\|_{T^{\frac{p}n}_{\alpha,2}}$ and  $\|F_1\|_{T^{\frac{p}n}_{\alpha,2}}$ have the same estimate than $A_1$, so we deduce that \eqref{eqn:estimate:A} holds.
	\vspace*{6pt}
	
	\noindent{\sf Estimate of $B$:}
	\begin{align*}
		B
		&\lesssim\|\bigl((T_g^{\ell+1}f) g^{m\ell}(T_g^{\ell}f)^{n-2} T_g^{\ell-1}f   \bigr)'\|_{T^{\frac{p}n}_{\alpha,2}}    
		+\|T_g^{\ell+1}f\, \bigl(g^{m\ell}(T_g^{\ell}f)^{n-2}T_g^{\ell-1}f   \bigr)'\|_{T^{\frac{p}n}_{\alpha,2}} \\
		&\lesssim B_1 +\|F_1\|_{T^{\frac{p}n}_{\alpha,2}} +\|F_2\|_{T^{\frac{p}n}_{\alpha,2}}  +B_2,
	\end{align*}
	where $F_1$ and $F_2$ are defined by \eqref{eqn:def:Fk}, 
	$B_1=\|(T_g^{\ell+1}f) g^{m\ell}(T_g^{\ell}f)^{n-2}T_g^{\ell-1}f  \|_{\alpha,\frac{p}n}$ and $B_2=\|(T_g^{\ell+1}f) g^{m\ell}(T_g^{\ell}f)^{n-2}(T_g^{\ell-1}f)'\|_{T^{\frac{p}n}_{\alpha,2}}$.
	
	\noindent{\sf Estimate of $B_1$:} Note that 
	\begin{equation*}
		B_1=
		\||T_g^{\ell+1}f| |g|^{n\tau\ell}|T_g^{\ell}f|^{n-2}|T_g^{\ell-1}f|\|_{\alpha,\frac{p}n}
		=\||Q_g^{\tau,\ell+1}f|
		|Q_g^{\tau,\ell}f|^{n-2}|Q_g^{\tau,\ell-1}f| \|_{\alpha,\frac{p}n},
	\end{equation*}
	so, by applying H\"{o}lder's inequality with exponents $p_1=n$, $p_2=\frac{n}{n-2}$ and $p_3=n$ (recall that $n>2$, since $n>1+\frac{2s}{\tau}$), and using the 
	estimate \eqref{eqn:upper:estimate:norm:Q:Hp}, we get that
	\begin{align}
		B_1
		&\le \|Q_g^{\tau,\ell+1}f\|_{\alpha,p}\, 
		\|Q_g^{\tau,\ell}f\|^{n-2}_{\alpha,p}\,
		\|Q_g^{\tau,\ell-1}f\|_{\alpha,p} 
		\\
		\nonumber
		&\le\opnorm{Q_g^{\tau,\ell+1}}_{\alpha,p}\,
		\opnorm{Q_g^{\tau,\ell}}^{n-2}_{\alpha,p}\,
		\opnorm{Q_g^{\tau,\ell-1}}_{\alpha,p} 
		\\
		\nonumber
		&\lesssim  \opnorm{Q_g^{\tau,\ell+1}}_{\alpha,p}\, 
		\|g\|_{\B^s_{\alpha}}^{s\ell(n-\frac{\ell+1}{\ell})}.
	\end{align}

	\noindent{\sf Estimate of $B_2$ for $\ell=1$:} Note that, in this case,
	\begin{equation*}
		B_2=\||T_g^2f|\,|g|^{n\tau}\,|T_gf|^{n-2}\,|f'|\|_{T^{\frac{p}n}_{\alpha,2}}
		=\||Q^{\tau,1}_gf|^{n-2}\,|Q^{\tau,2}_gf|\,|f'|\|_{T^{\frac{p}n}_{\alpha,2}}.
	\end{equation*}
	Since $|Q^{\tau,j}_gf|=|h_j|^{\frac1n}$, for $j=1,2$, 
	we have that 
	${\mathcal M}(Q^{\tau,j}_gf)=({\mathcal M}h_j)^{\frac1n}$, where ${\mathcal M}$ is the nontangential maximal operator defined by~{\eqref{eqn:definition:non-tangential:maximal:function}}, and so \eqref{eqn:boundedness:non-tangential:maximal:function:1} shows that 
	\begin{equation}
		\label{eqn:boundedness:non-tangential:maximal:operator}
		\|{\mathcal M}(Q^{\tau,j}_gf)\|_{\alpha,q}
		=\|{\mathcal M}h_j\|_{\alpha,\frac{q}n}^{\frac1n} 
		\lesssim \|h_j\|_{\alpha,\frac{q}n}^{\frac1n}
		=\|Q^{\tau,j}_gf\|_{\alpha,q}
		\le\opnorm{Q^{\tau,j}_g}_{\alpha,q},
	\end{equation}
	where~{\eqref{eqn:h_j:Q}} gives the last identity.
	Thus we estimate $B_2$ as
	follows
	\begin{align*}
		B_2
		&=\biggl\{\int_{\overline{\D}}\biggl(\int_{\Gamma(\zeta)}|Q^{\tau,1}_gf|^{2(n-2)}\,|Q^{\tau,2}_gf|^2\,|f'|^2\,dA\biggr)^{\frac{p}{2n}}dA_{\alpha}(\zeta)\biggr\}^{\frac{n}p}
		\\
		&\le\biggl\{\int_{\overline{\D}}
		({\mathcal M}(Q^{\tau,1}_gf))^{(n-2)\frac{p}n}
		({\mathcal M}(Q^{\tau,2}_gf))^{\frac{p}n}
		\biggl(\int_{\Gamma(\zeta)}|f'|^2\,dA\biggr)^{\frac{p}{2n}}dA_{\alpha}(\zeta)\biggr\}^{\frac{n}p}.
	\end{align*}
	Then H\"{o}lder's inequality with exponents $p_1=\frac{n}{n-2}$, $p_2=n$, and $p_3=n$,  and estimates~{\eqref{eqn:boundedness:non-tangential:maximal:operator}}, \eqref{eqn:estimate:Apalpha:norm:derivative}, and~{\eqref{eqn:upper:estimate:norm:Q:Hp}} imply that
	\begin{align*}
		B_2 
		&\le \|{\mathcal M}(Q^{\tau,1}_gf)\|^{n-2}_{\alpha,p}\,
		\|{\mathcal M}
		(Q^{\tau,2}_gf)\|_{\alpha,p}\,\|f'\|_{T^p_{\alpha,2}}
		\\ 
		\nonumber    
		&\lesssim \opnorm{Q^{\tau,1}_g}_{\alpha,p}^{n-2}\,
		\opnorm{Q^{\tau,2}_g}_{\alpha,p}\,\|f\|_{\alpha,p}
		\lesssim\opnorm{Q^{\tau,2}_g}_{\alpha,p}\,\|g\|^{s(n-2)}_{\B^s_{\alpha}}.
	\end{align*}

	\noindent{\sf Estimate of $B_2$ for $\ell>1$:}
	Since $B_2=\||T_g^{\ell+1}f| |g|^{n\tau\ell}|T_g^{\ell}f|^{n-2}|(T_g^{\ell-1}f)'|   \|_{T^{\frac{p}n}_{\alpha,2}}$,
	\begin{equation*}
		B_2
		=\|(|g|^{\tau}|g'|)|h_{\ell+1}|^{\frac1n}
		|h_{\ell}|^{\frac{n-2}n}|h_{\ell-2}|^{\frac1n}\|_{T^{\frac{p}n}_{\alpha,2}(\nu)}
		=\||h_{\ell+1}|^{\frac1n}
		|h_{\ell}|^{\frac{n-2}n}|h_{\ell-2}|^{\frac1n}\|_{T^{\frac{p}n}_{\alpha,2}(\nu)},
	\end{equation*}
	where $\nu$ is the measure $\nu_{g,s}$ defined by~{\eqref{eqn:def:nugs}}. Then Proposition~{\ref{prop:Holder:ineq:tent:spaces}}
	(with exponents $p_1=p$, $p_2=\tfrac{p}{n-2}$,    
	$p_3=p$, and $q_1=q_2=q_3=6$), identity~{\eqref{eqn:Tpq:norm:of:a:power}},   and estimates~{\eqref{eqn:Mpq:simeq:BMOAs-norm}}, \eqref{eqn:h_j:Q}, and~{\eqref{eqn:upper:estimate:norm:Q:Hp}}  give
	\begin{align*}	
		B_2
		&\le\||h_{\ell+1}|^{1/n}\|_{T^{p_1}_{\alpha,q_1}(\nu)}\,
		\||h_{\ell}|^{(n-2)/n}\|_{T^{p_2}_{\alpha,q_2}(\nu)}\,
		\||h_{\ell-2}|^{1/n}\|_{T^{p_3}_{\alpha,q_3}(\nu)}
		\\
		&=\|h_{\ell+1}\|_{T^{p/n}_{\alpha,6/n}(\nu)}^{1/n}\,
		\|h_{\ell}\|_{T^{p/n}_{\alpha,6(n-2)/n}(\nu)}^{(n-2)/n}\,
		\|h_{\ell-2}\|_{T^{p/n}_{\alpha,6/n}(\nu)}^{1/n}
		\\
		&\lesssim\|g\|^s_{\B^s_{\alpha}}\,
		\|h_{\ell+1}\|_{\alpha,p/n}^{1/n}\,
		\|h_{\ell}\|_{\alpha,p/n}^{(n-2)/n}\,
		\|h_{\ell-2}\|_{\alpha,p/n}^{1/n}
		\\
		&\lesssim 
		\|g\|^s_{\B^s_{\alpha}}\,
		\opnorm{Q_g^{\tau,\ell+1}}_{\alpha,p}\, 
		\opnorm{Q_g^{\tau,\ell}}^{n-2}_{\alpha,p}\,
		\opnorm{Q_g^{\tau,\ell-2}}_{\alpha,p}
		\\      
		&\lesssim
		\opnorm{Q_g^{\tau,\ell+1}}_{\alpha,p}\,
		\|g\|^{s\ell(n-\frac{\ell+1}{\ell})}_{\B^s_{\alpha}}.
	\end{align*}
	Moreover, Lemma \ref{lem:key:tool:lower:estimate:norm:word} shows that $\|F_1\|_{T^{\frac{p}n}_{\alpha,2}}$ and  $\|F_2\|_{T^{\frac{p}n}_{\alpha,2}}$ have the same estimate than $B_j$, $j=1,2$, so it follows that \eqref{eqn:estimate:B} holds. And that ends the proof.
\end{proof}

\section{Proof of Theorem \ref{thm1:boundedness:g-operators} }
\label{section:proofs:of:thm1.3-1.4:proposition1.5}

\begin{proof}[{\bf Proof of Theorem \ref{thm1:boundedness:g-operators}}]
Assume that $g\in\B^s_{\alpha}$.  Then Theorem~{\ref{thm:sharp:estimate:norm:word} {\sf b)}}  shows that $S_g^mT_g^n\in\BB(A^p_{\alpha})$. Moreover, the estimates \eqref{eqn:key:estimate:Carleson:norms:symbols1} and \eqref{eqn:key:estimate:Carleson:norms:symbols2}
give that  $g\in\B^{s'}_{\alpha}$, for any $1\le s'\le s$. Since $\frac{m-j}{n_j}\le\frac{m}n$, for $j=1,\dots,m$, we have that 
$1\le s(j,k):=\frac{m-j}{n_j+k}+1\le s$, for $j=1,\dots,m$ and $k\in\N_0$, and so $g\in\B^{s(j,k)}_{\alpha}$. By applying  Theorem~{\ref{thm:sharp:estimate:norm:word} {\sf b)}} again, we obtain that 
$S^{m-j}_gT^{n_j}_gP_j(T_g)\in\BB(A^p_{\alpha})$, for $j=1,\dots,m$, and therefore $L\in\BB(A^p_{\alpha})$.

Now assume that $L\in\BB(A^p_{\alpha}(0))$, and we want to prove that $g\in\B^s_{\alpha}$, or equivalently $\sup_{0<r<1}\|g_r\|_{\B^s_{\alpha}}<\infty$, by 
Proposition~{\ref{prop:radialized:symbol}} and Corollary~{\ref{cor:prop:equiv:Garsia:Carleson:symbols}}.
Note that, without loss of generality, we may assume that $n_j\le n$, for $j=1,\dots,m$. Indeed, if $n_j>n$ then we may replace $T^{n_j}_gP_j(T_g)$ by $T^n_gQ_j(T_g)$ in~{\eqref{eqn:thm1:boundedness:g-operators}}, where $Q_j(z)=z^{n_j-n}P_j(z)$, and note that $\frac{m-j}n<\frac{m}n$.  
Then Theorem~{\ref{thm:sharp:estimate:norm:word}}, Proposition 4.3 in~{\cite{Aleman:Cascante:Fabrega:Pascuas:Pelaez}},  Proposition~{\ref{prop:boundedness:Tg}}
and estimate~{\eqref{eqn:key:estimate:Carleson:norms:symbols1}} show that
\begin{align*}
\|g_r\|^{m+n}_{\B^s_{\alpha}}
&\lesssim
\opnorm{S^m_{g_r}T^n_{g_r}}_{\alpha,p}
\lesssim
\opnorm{L_{g_r}}_{\alpha,p}
+\sum_{j=1}^m \|S^{m-j}_{g_r}T^{n_j}_{g_r}\|_{\alpha,p}\|P_j(T_{g_r})\|_{\alpha,p}\\
&\le
\opnorm{L_g}_{\alpha,p}
+\sum_{j=1}^m \|S^{m-j}_{g_r}T^{n_j}_{g_r}\|_{\alpha,p}\,\widetilde{P}_j(\|T_{g_r}\|_{\alpha,p})\\
&\lesssim
\opnorm{L_g}_{\alpha,p}
+\sum_{j=1}^m \|g_r\|^{m-j+n_j}_{\B^{\frac{m-j}{n_j}+1}_{\alpha}}\,\widetilde{P}_j(\opnorm{L_{g_r}}_{\alpha,p}^{\frac1{m+n}})\\
&\lesssim
\opnorm{L_g}_{\alpha,p}
+\sum_{j=1}^m \|g_r\|^{(m+n)\varepsilon_j}_{\B^s_{\alpha}}\,\widetilde{P}_j(\opnorm{L_g}_{\alpha,p}^{\frac1{m+n}}),
\end{align*}
where $\widetilde{P}_j(z)=\sum_{k=0}^N|a_k|z^k$ if $P(z)=\sum_{k=0}^Na_kz^k$, and $\varepsilon_j=\frac{m-j+n_j}{m+n}$.

Since $n_j\le n$, we have that $0<\varepsilon_j<1$, for $j=1,\dots,m$, and hence we conclude that $\sup_{0<r<1}\|g_r\|_{\B^s_{\alpha}}<\infty$, which ends the proof.
\end{proof}

\end{document}